%% file: RealizationSurvey.tex
\documentclass[10pt]{article}
\usepackage{amsfonts}
\usepackage{graphicx}
\usepackage{amsmath}
\usepackage{amsthm}
\usepackage{amssymb}
\usepackage{amscd}
\usepackage{color}
\usepackage{url}
\usepackage{xypic}
\usepackage[all]{xy}

\usepackage{graphicx}
\usepackage{microtype}
\usepackage{enumerate}
\usepackage{pinlabel}
\usepackage[top=1.1in,bottom=1.4in,left=1.3in,right=1.3in]{geometry}

\newtheorem{thm}{Theorem}

\newtheorem{prop}[thm]{Proposition}

\theoremstyle{definition}
\newtheorem{defi}[thm]{Definition}
\newtheorem{prob}[thm]{Problem}

\newtheorem{qu}[thm]{Question}

\newtheorem{const}[thm]{Construction}

\theoremstyle{definition}
\newtheorem{app}[thm]{Application}

\renewcommand{\H}{\mathbb{H}}
\newcommand{\R}{\mathbb{R}}
\newcommand{\Z}{\mathbb{Z}}
\newcommand{\N}{\mathbb{N}}
\newcommand{\Q}{\mathbb{Q}}
\newcommand{\C}{\mathbb{C}}
\newcommand{\D}{\mathbb{D}}
\newcommand{\T}{\mathbb{T}}
\newcommand{\ra}{\rightarrow}
\newcommand{\Si}{\Sigma}

\newcommand{\de}{\delta}

\newcommand{\Ga}{\Gamma}
\newcommand{\til}{\tilde}

\newcommand{\al}{\alpha}
\newcommand{\pa}{\partial}
\newcommand{\xra}{\xrightarrow}
\newcommand{\hra}{\hookrightarrow}
\newcommand{\ca}{\mathcal}
\newcommand{\sbs}{\subset}
\newcommand{\ld}{\ldots}
\newcommand{\ti}{\times}

\newcommand{\Br}{\mathrm{Br}}

\DeclareMathOperator{\id}{id}

\DeclareMathOperator{\fix}{fix}

\DeclareMathOperator{\PSL}{PSL}
\DeclareMathOperator{\SL}{SL}
\DeclareMathOperator{\SO}{SO}
\DeclareMathOperator{\GL}{GL}
\DeclareMathOperator{\Homeo}{Homeo}
\DeclareMathOperator{\Diff}{Diff}

\DeclareMathOperator{\Mod}{Mod}

\DeclareMathOperator{\Conf}{Conf}
\DeclareMathOperator{\QC}{QC}
\DeclareMathOperator{\QS}{QS}
\DeclareMathOperator{\Aff}{Aff}
\DeclareMathOperator{\Aut}{Aut}

\DeclareMathOperator{\Symp}{Symp}

\title{Realization problems for diffeomorphism groups} 

\author{Kathryn Mann and Bena Tshishiku}

\date{}

\begin{document}

\maketitle

\begin{abstract}
We discuss recent results and open questions on the broad theme of (Nielsen) realization problems.  Beyond realizing subgroups of mapping class groups, there are many other natural instances where one can ask if a surjection from a group of diffeomorphisms of a manifold to another group admits a section over particular subgroups.   This survey includes many open problems, and some short proofs of new results that are illustrative of key techniques; drawing attention to parallels between problems arising in different areas.
\end{abstract} 

\numberwithin{thm}{section}
\numberwithin{equation}{section}

\setcounter{section}{0}

\input{Intro.tex}

\input{Cohomological.tex}

\input{Dynamical.tex}

\input{Positive1.tex}

\input{RAAGs.tex}

\bibliographystyle{alpha}
\small
\bibliography{biblio}

\end{document}

%% file: Intro.tex
\section{Introduction}

One way to understand the algebraic structure of a large or complicated group is to surject it to a simpler group, then measure what is lost.   In other words, given $G$, we take an exact sequence $1\ra K \to G \to H\ra 1$ and study how much $H$ differs from $G$.  This difference can be measured through the obstructions to a (group-theoretic) \emph{section} $\phi: H \to G$, as well as the obstructions to sections over subgroups of $H$. 

In the case where $G$ is the group of homeomorphisms or diffeomorphisms of a manifold, this question has a long history and important interpretations, both from the topological point of view of flat bundles, and the dynamical point of view through group actions of $M$.  
This article is intended as an invitation to section problems for diffeomorphism groups; illustrating applicable techniques and listing open problems.
To introduce this circle of ideas, we begin with the classical case of surface bundles and diffeomorphism groups of surfaces.  
  
\paragraph{The flatness problem for surface bundles.}
Let $\Si$ be a smooth, oriented closed surface and let $E\ra B$ be a fiber bundle with fiber $\Si$ and structure group the group $\Diff^r(\Si)$ of orientation-preserving $C^r$ diffeomorphisms of $\Si$, for some fixed $r \geq 0$.  Such a bundle is said to be \emph{flat} or \emph{foliated} of class $C^r$ if it admits a $C^r$ foliation transverse to the fibers whose leaves project to the base as covering spaces.  Equivalently, $E\ra B$ is flat if its structure group can be taken to be $\Diff^r(\Si)$ with the discrete topology. In a flat bundle, parallel transport along the leaves of the foliation defines a holonomy representation $\phi:\pi_1(B)\ra\Diff^r(\Si)$. 

It is a basic problem to determine when a surface bundle is flat; and even in the case where the base is a surface, it is a open question whether every surface bundle admits a flat structure.  A foundational result of Earle--Eells \cite{EE} states that, when $\chi(\Si)<0$, a $\Si$-bundle $E\ra B$ is determined up to bundle isomorphism by its \emph{monodromy representation} $\rho:\pi_1(B)\ra\Mod(\Si)$.  Here $\Mod(\Si) := \pi_0(\Diff^r(\Si))$ denotes the {\em mapping class group} of $\Si$ (which is independent of $r$).  Thus, $E\ra B$ admits a $C^r$ flat structure if and only if its monodromy lifts to $\Diff^r(\Si)$, as in the diagram below.  In this case, the lift agrees (up to conjugacy) with the holonomy $\phi$ mentioned above.
\[\begin{xy}
(0,12)*+{\Diff^r(\Si)}="A";
(-20,0)*+{\pi_1(B)}="B";
(0,0)*+{\Mod(\Si)}="C";
{\ar@{-->}"B";"A"}?*!/_3mm/{\til\rho};
{\ar "A";"C"}?*!/_3mm/{\pi};
{\ar "B";"C"}?*!/^3mm/{\rho};
\end{xy}\]

The flat bundle question motivates the following problem. Let $\Si_{g,m}^b$ denote a compact surface of genus $g$ with $b$ boundary components and $m$ marked points. Define $\Diff^r(\Si_{g,m}^b)$ as the group of orientation-preserving diffeomorphisms of $\Si$ that preserve the marked points as a set, and define $\Diff^r_\pa(\Si_{g,m}^b)<\Diff^r(\Si_{g,m}^b)$ as the subgroup of diffeomorphisms that restrict to the identity on the boundary $\pa\Si$.  The mapping class group $\Mod_{g,m}^b:=\pi_0 \Diff^r_\pa(\Si_{g,m}^b)$ can be viewed as the group of diffeomorphism of $\Si_{g,m}^b$ modulo isotopies that fix the boundary and the punctures\footnote{following standard convention, we let $\Mod_g$ denote $\Mod_{g, 0}^0$}. 

\begin{prob}[Generalized Nielsen realization]\label{prob:lifting-mod}
Let $\Gamma$ be a finitely-generated group, $0 \leq r \leq \infty$, and let $\rho:\Gamma\ra\Mod_{g,m}^b$ be a homomorphism. Does there exist a homomorphism $\tilde\rho:\Gamma\rightarrow\Diff^r_\pa(\Si_{g,m}^b)$ so that $\pi\circ\til\rho=\rho$?
\end{prob}

If the answer is yes, we say $\rho$ is \emph{realized by $C^r$ diffeomorphisms}, or by \emph{homeomorphisms} in the case $r=0$.  (We will use both $\Diff^0(M)$ and $\Homeo(M)$ as notation for the homeomorphism group of a manifold $M$.)  The map $\tilde{\rho}$ is called a \emph{section} of $\Diff^r(\Si_{g,m}^b) \to \Mod(\Si_{g,m}^b)$ over $\rho$, so we will also often refer to this and related questions as \emph{section problems}.    Earle--Eells' result generalizes to surfaces with boundary and marked points \cite{ES,hatcher-mw}, and from the bundle perspective, marked points correspond to preferred sections.   

As is well known, Problem \ref{prob:lifting-mod} was answered positively for finite, cyclic groups by Nielsen. Nielsen's original motivation was different from ours -- the question he answered (positively) was: \emph{if $f$ is a diffeomorphism of $\Si$ such that $f^n$ is isotopic to the identity, is $f$ isotopic to some $g \in \Diff^\infty(\Si)$ with $g^n = \id$?} Problem \ref{prob:lifting-mod} for finite groups $\Gamma \subset \Mod_g$ eventually became known as Nielsen's problem.   Fenchel \cite{Fenchel} gave a positive answer for finite {\em solvable} groups, and the general result for finite groups is due to Kerckhoff \cite{Kerckhoff}, with alternate proofs by Tromba, Gabai and Wolpert  \cite{Tromba, Gabai, Wolpert} appearing later.  Going forward in this paper we will 
focus on realization problems for infinite groups.  

The case $\Gamma = \Mod_g$ for $\tilde{\rho}$ with image in $\Homeo(\Sigma_g)$ was posed by Thurston in Kirby's problem list \cite[Prob.\ 2.6]{Kirby-problems}.  This was resolved (negatively) in case of $C^2$ diffeomorphisms by Morita \cite{Morita} for $g \geq 18$, in the $C^1$ case by Franks--Handel \cite{FH} for $g \geq 3$ -- whose proofs we sketch later on -- and then by Markovic \cite{Markovic} and Markovic--Saric \cite{MS} for homeomorphisms and $g \geq 2$. Although Morita's approach generalizes to finite-index subgroups, and Franks--Handel cover several other special cases, for most infinite subgroups $\Gamma<\Mod_{g,m}^b$ Problem \ref{prob:lifting-mod} remains completely open (see discussion in \cite[\S 6.3]{Farb06}).  

One source of difficulty in Problem \ref{prob:lifting-mod} is our poor understanding of finitely-generated subgroups of diffeomorphism groups.  For instance, the following problem is open. 

\begin{prob} \label{prob:subgroup_homeo}
Give an example of a finitely-generated, torsion free group $\Gamma$, and a surface $S$, such that $\Gamma$ is not isomorphic to a subgroup of $\Homeo(S)$.   
\end{prob}

Of course, it suffices to give an example for $\Homeo_\pa(\D^2)$ as this embeds in $\Homeo(S)$ -- and in fact in the identity component $\Homeo_0(S)$ -- for any other surface.   Indeed, the difficulty in the problem is understanding the algebraic structure of $\Homeo_0(S)$; replacing $\Homeo(S)$ by $\Mod(S)$ renders it a reasonable exercise.  
Problem \ref{prob:subgroup_homeo} is also completely open for homeomorphism groups of manifolds of higher dimension. In general the regularity of diffeomorphisms can make a great difference, and Problem \ref{prob:subgroup_homeo} has been recently solved for $C^2$ diffeomorphisms with the solution to the Zimmer conjecture \cite{BFH}, with examples given by higher rank lattices.   However, we are far from a complete understanding of the algebraic structure of groups that can and cannot act.  

\paragraph{Nielsen realization problems in higher dimensions.} 
Problem \ref{prob:lifting-mod} can be posed for manifolds $M$ with $\dim M\ge3$. Does $\Diff^r(M) \to \pi_0\big(\Diff^r(M)\big)$ admit a section? Which finitely-generated subgroups can be realized?  Unlike in the surface case, the group $\pi_0(\Diff^r(M))$ is not always finitely generated and its algebraic structure may depend on $r$. Spheres are a particularly interesting class of examples:

\begin{qu}[A.\ Kupers]
For $n \geq 5$, the group $\pi_0(\Diff^\infty(S^n))$ is the group of homotopy $(n+1)$-spheres, a finite abelian group.  Can this group be realized by diffeomorphisms of $S^n$?  
\end{qu}
Some cases are covered by Schultz \cite{schultz}. Other instances of Nielsen realization problems for finite subgroups of mapping class groups are discussed by Block--Weinberger in \cite{BW}, who prove several non-realizability results and give a good survey of the state of the art.   Examples of non-realizability results for {\em infinite}, finitely-generated subgroups for higher-dimensional manifolds are given in \cite{tshishiku-nielsen}; see Theorem \ref{thm:tshishiku-nielsen} below.


\paragraph{Beyond mapping class groups.}  

When $M$ is a closed manifold, the path-component of the identity $\Diff^r_0(M)\subset\Diff^r(M)$ is simple, provided that $r \neq \dim(M) +1$. For $r>0$, this is a deep result of Mather \cite{Mather1, Mather2}, and for $r=0$ it follows from \cite{Anderson} and \cite{EK} (it is open for $r=\dim(M)+1$). Thus, realization problems for subgroups of mapping class groups are essentially the only section problems for these groups.  However, when $M$ has boundary, there is an obvious map $\Diff^r(M) \overset{R}\to \Diff^r(\pa M)$.  This may not be surjective, but is when restricted to the identity components $\Diff^r_0(M) \to \Diff^r_0(\pa M)$. In \cite{ghys91}, Ghys asked 
in which cases this map admits a section.
In parallel with Problem \ref{prob:lifting-mod}, we ask: 

\begin{prob} \label{prob:boundary}
Let $\Gamma$ be a finitely-generated group, $0 \leq r \leq \infty$, and let $\rho:\Gamma\ra \Diff^r_0(\pa M)$ be a homomorphism. In which cases does there exist a homomorphism $\tilde\rho:\Gamma\rightarrow\Diff^r(M)$ so that $R \circ\til\rho=\rho$?
\end{prob}

Rephrased in terms of flat bundles, Problem \ref{prob:boundary} asks whether a flat $(\pa M)$-bundle $E$ is the boundary of a flat $M$-bundle.

As a further variant on Problems \ref{prob:lifting-mod} and \ref{prob:boundary}, one can also restrict to \emph{subgroups} of $\Diff^r(M)$ and pose the same family of questions. For real-analytic diffeomorphisms, 
Problem \ref{prob:lifting-mod} for $\Gamma = \Mod_g$ was answered negatively in \cite{CC}.  Bowden \cite{Bowden} discusses Problem \ref{prob:boundary} for the boundary map on the group of surface symplectomorphisms 
$\Symp(\Si_g^1) \to \Diff_0(S^1)$, and Epstein--Markovic discuss the boundary map from quasi-conformal homeomorphisms of the disk to quasi-symmetric homeomorphisms of the circle.  Interestingly, this latter problem is also related to Nielsen realization for the mapping class group of a surface; see Section \ref{sec:continuity} below.


\paragraph{Germs and holonomy of foliations.}
Considering foliations rather than foliated bundles leads to a class of related problems.   Let $\Diff^r(\R^n, 0)$ denote the group of diffeomorphisms of $\R^n$ fixing the origin, and $\mathcal{G}^r(k)$ the group of germs at the origin of elements of $\Diff^r(\R^n, 0)$.  Analogous to the holonomy of a flat bundle, when a manifold $M$ admits a transversely-$C^r$ codimension-$k$ foliation $\mathcal{F}$, each leaf $L$ has a holonomy representation  $\pi_1(L)\ra \mathcal{G}^r(k)$.  Thus, understanding the structure of $\mathcal{G}^r(k)$ and its finitely-generated subgroups plays an important role in the study of the structure of foliations. 

However, the algebraic structure $\mathcal{G}^r(k)$ is not well understood.  For instance, the following problem is open.
\begin{qu} \label{q:germs_iso}
Suppose $\phi: \mathcal{G}^r(n) \to \mathcal{G}^s(m)$ is an isomorphism.  Is it necessarily true that $m = n$, $r=s$, and $\phi$ is an inner automorpshim?
\end{qu} 

\noindent The corresponding (positive) result for isomorphisms $\Diff^r(M) \to \Diff^s(M)$, where $M$ is a compact manifold, is a theorem of Filipkiewicz \cite{Filipkiewicz}.  

Finitely-generated subgroups of $\mathcal{G}^r(n)$ are also poorly understood.  Navas \cite{DNR} asked whether every finitely-generated subgroup of $\mathcal{G}^0(2)$ is isomorphic to a subgroup of $\Diff^0(\R^2)$. While it is expected that the answer is negative, the question was motivated by a positive result in the case of $\R^1$; see \cite[\S 1.1.3]{DNR}. 
Because of the natural surjection $\Diff^r(\R^n, 0) \to \mathcal{G}^r(n)$, the study of groups of germs lends itself to section problems. 

\begin{prob} \label{prob:germ}
Which finitely-generated groups are subgroups of $\mathcal{G}^r(n)$?  What are obstructions to lifting a representation $\rho: \Gamma \to \mathcal{G}^r(n)$ to $\Diff^r(\R^n, 0)$?  
\end{prob}

Thurston's \emph{stability theorem} \cite{Thurston} states that finitely-generated subgroups of $\mathcal{G}^1(n)$ with trivial derivative at $0$ necessarily have the property that they admit surjective homomorphisms to $\Z$ (i.e. $\mathcal{G}^1(n)$ is a {\em locally indicable} group).  While Thurston's motivation for this result was a generalization of the Reeb stability theorem for foliations, we will see numerous applications of this stability theorem to other problems -- including approaches to Problem \ref{prob:lifting-mod} -- described in \S\ref{sec:dynamical}.  One is optimistic that any further answer to the first part of Problem \ref{prob:germ} would have a similar range of applications.

\paragraph{Organization of paper.}

In order to draw parallels between the various problems described above, we have organized this paper by \emph{technique} (broadly categorized as dynamical or cohomological) rather than by results. In \S\ref{sec:Cohomological} we discuss examples where cohomology of groups and classifying spaces has been successfully used to obstruct lifting problems. \S\ref{sec:dynamical} contains a discussion of dynamical techniques used to provide other obstructions. In \S\ref{sec:positive} we give some positive results for Problem \ref{prob:lifting-mod}, i.e.\ interesting groups that \emph{can} be realized by diffeomorphisms. 
Questions and problems have been attributed to their authors when known, although we did not make an exhaustive attempt to trace them to their original sources.

\paragraph{Acknowledgements.}
The authors would like to thank B. Farb, N. Salter and A. Kupers for feedback on this work.  Theorem \ref{thm:germ_nonrealize} came up in discussion with F. Le Roux.  
K.M.\ was partially supported by NSF grant DMS 1606254, and B.T.\ was partially supported by NSF grant DMS 1502794.

%% file: Cohomological.tex
\section{Cohomological techniques}  \label{sec:Cohomological}

In this section we take an algebro-topological approach to lifting problems.  In its most basic form, the general approach is as follows:
Given a surjection $\pi: G \to H$ and a homomorphism $\rho:\Gamma\ra H$, a lift $\tilde{\rho}$ as in the figure on the left below induces maps on group cohomology as indicated on the right.

\[\begin{xy}
(0,12)*+{G}="A";
(-20,0)*+{\Gamma}="B";
(0,0)*+{H}="C";
{\ar@{-->}"B";"A"}?*!/_3mm/{\tilde\rho};
{\ar "A";"C"}?*!/_3mm/{\pi};
{\ar "B";"C"}?*!/^3mm/{\rho};
\end{xy}
\hspace{1in}\begin{xy}
(0,12)*+{H^*(G)}="A";
(-25,0)*+{H^*(\Gamma)}="B";
(0,0)*+{H^*(H)}="C";
{\ar@{-->}"A";"B"}?*!/^3mm/{\tilde\rho^*};
{\ar "C";"A"}?*!/^3mm/{\pi^*};
{\ar "C";"B"}?*!/_3mm/{\rho^*};
\end{xy}\]

Classes in the kernel of $\pi^*$ serve as obstructions to lifts.  Specifically, if $c \in \ker\pi^*$, and $\rho$ is a representation such that $\rho^*(c)\neq0\in H^*(\Gamma)$, then no map $\tilde \rho$ can make the diagram commute.  

While this approach (or variations of it) can be successful, 
in general it can be difficult to find a class $c$ that works, and the absence of cohomological obstructions is not sufficient to guarantee the existence of a lift. For an example, compare Theorems \ref{thm:Nariman} and \ref{thm:st} below. However, one benefit of the cohomological approach (when it works) is that often one can conclude that no realization exists even after passing to a finite-index subgroup of $\Gamma$. 

\subsection{The Euler class as an obstruction}

We begin with a simple example of the strategy introduced above. 
Fix $h\ge2$.  Then $\pi_1(\Si_h)$ is a discrete, subgroup of $\PSL(2,\R)$, which acts on $S^1$ by M\"obius transformations.  Let $\rho:\pi_1(\Si_h)\ra\PSL(2,\R)\ra \Homeo^+(S^1)$ denote this homomorphism.  

\begin{prop}\label{prop:boundary-extend}
Fix $g\ge1$, and let $\Si_{g}^{1}$ be a genus $g$ surface with 1 boundary component. Then the action $\rho$ of $\pi_1(\Si_h)$ on $\partial \Si_{g}^{1}\simeq S^1$ does not extend to a representation into $\Homeo_0(\Si_{g}^{1})$. 
\end{prop}

Here, as always, $\Homeo_0$ denotes the identity component of the group.  As we remark after the proof, the situation is much more subtle (and the proposition false) if $\Homeo_0$ is replaced by $\Homeo$ in the statement above. 

The proof of Proposition \ref{prop:boundary-extend} is easy when framed in the language of characteristic classes and classifying spaces.
For the reader unacquainted with this theory, we summarize the main properties we need before giving the proof.  See \cite[Ch.\ 4]{morita-book} for more details.

\paragraph{Classifying space basics.} Every topological group $G$ has a classifying space $BG$ that classifies principal $G$-bundles. If $G<\Homeo(X)$, then $BG$ also classifies $X$-bundles with structure group $G$. If $G$ has the discrete topology (indicated by $G^\de$), then $BG^\de\sim K(G,1)$ is an \emph{Eilenberg--Maclane space} for $G$ (a connected CW-complex with fundamental group $G$ and contractible universal cover). In particular $B\Diff(\Si_g)$ classifies $\Si_g$-bundles and $B\Diff(\Si_g)^\de$ classifies flat $\Si_g$-bundles. 

The classifying space construction $G\mapsto BG$ is functorial: a continuous homomorphism $G\ra H$ induces a continuous map $BG\ra BH$. Furthermore, if $G\ra H$ is a homotopy equivalence, then $BG\ra BH$ is also a homotopy equivalence, and a short exact sequence $1\ra K\ra G\ra H\ra1$ induces a fibration $BK\ra BG\ra BH$.

In the special case where $G = \Homeo(S^1)$, the inclusion $\SO(2)\hra\Homeo_+(S^1)$ induces a homotopy equivalence, so the classifying space $B\Homeo_+(S^1)$ is homotopy equivalent to $B\SO(2)\sim \mathbb{C}P^\infty$. As such $H^*(B\Homeo(S^1);\Z)$ is a polynomial ring $\Z[e]$ generated by a class $e\in H^2(B\Homeo(S^1))$ which has the following property: a circle bundle $E\ra B$ classified by a map $f_E:B\ra B\Homeo(S^1)$ admits a continuous section if and only if $f_E^*(e)=0$.

\begin{proof}[Proof of Proposition \ref{prop:boundary-extend}]
If $\rho$ extends to $\tilde\rho:\pi_1(\Si_h)\ra\Homeo_0(\Si_{g}^{1})$, then there is a diagram that commutes up to homotopy:
\[\begin{xy}
(0,15)*+{B\Homeo_0(\Si_{g}^{1})}="A";
(-30,0)*+{\Si_h}="B";
(0,0)*+{B\Homeo_+(S^1)}="C";
(-20,-3)*+{B\rho}="D";
{\ar@{-->}"B";"A"}?*!/_3mm/{B\tilde\rho};
{\ar "A";"C"}?*!/_3mm/{B\pi};
{\ar "B";"C"}?*!/^3mm/{};
\end{xy}\]
Let $e(\rho):=B\rho^*(e)\in H^2(\Si_h)$ be the Euler class of $\rho$. On the one hand, $e(\rho)$ is the Euler class of the unit tangent bundle $T^1\Si_h\ra\Si_h$, which is nonzero because $\chi(\Si_h)\neq0$. On the other hand, $\Homeo_0(\Si_{g}^{1})$ is contractible, so $B\Homeo_0(\Si_{g}^{1})$ is also contractible, and this implies that $\pi^*(e)=0$. Thus $\til\rho$ does not exist. 
\end{proof}

The above argument shows that no action $\rho:\pi_1(\Si_h)\ra\Homeo_+(S^1)$ with $e(\rho)\neq0$ extends to $\Homeo_0(\Si_g^1)$ for $g\ge1$.  There are many such actions, even among representations into $\PSL(2,\R)$.   However, one can produce actions $\rho:\pi_1(\Si_h)\ra\Homeo(S^1)$ with $e(\rho)\neq0$ that \emph{do} extend to $\Homeo(\Si_{g}^{1})$ for any $g\geq 2$. This uses a ``stabilization'' argument \cite{Bowden}:  start with any $\rho':\pi_1(\Si_{h'})\ra\Homeo(S^1)$, present $\pi_1(\Si_{h'}) = \langle a_1,\ld,b_{h'} \mid \prod_{i=1}^{h'} [a_i, b_i]=1 \rangle$, and choose $\tilde \rho'(a_i), \tilde \rho'(b_i)\in\Homeo(\Si_g^1)$ extending $\rho'(a_i)$ and $\rho'(b_i)$ on $\pa\Si_g^1$ (technically we want extensions that are a trivial product on a collar neighborhood of $\pa\Si_g^1$). Then $f:= \prod_{i=1}^h [\tilde \rho'(a_i), \tilde \rho'(b_i)]$ belongs to the group $\Homeo_{N(\pa)}(\Si_g^1)$ of homeomoprhisms that are the identity on a neighborhood of $\pa\Si_g^1$. This group is perfect 
\cite{Fischer}
\cite[\S5.1]{farb-margalit}, so we can write $f^{-1} = \prod_{j=1}^n[x_j, y_j]$.  Let $h = h'+n$, and extend $\rho'$ to 
\[\rho:\pi_1(\Si_h) = \langle a_i, b_i, x_j, y_j \mid \prod_{i=1}^{h'} [a_i, b_i] \prod_{j=1}^{n} [x_j, y_j]=1 \rangle\ra\Homeo_+(S^1)\]
by mapping $x_j$ and $y_j$ to the identity.  By design, $e(\rho)=e(\rho')$, and $\rho$ extends to $\Homeo(\Si_{g}^{1})$. Note that $e(\rho)\in\Z$ determines the class $[f]\in\pi_0\big[\Homeo_{N(\pa)}(\Si_g^1)\cap \Homeo_0(\Si_g^1)\big]\simeq\Z$. In particular $x_i,y_i\in\Homeo_0(\Si_g^1)$ is possible only if $e(\rho)=0$. This is consistent with Proposition \ref{prop:boundary-extend}.

\begin{qu}
Given a representation $\rho: \pi_1(\Si_{h'})\ra\Homeo_+(S^1)$ as above, what is the minimal stabilization genus $h$ required to extend the action to $\Homeo_+(\Si_{g}^{1})$?  Are there examples where this answer is different in the $C^0$ and $C^\infty$ cases?  
\end{qu}

The Milnor-Wood inequality (Theorem \ref{thm:milnor-wood} below) gives bounds on the value of the Euler class $\rho^*(e) \in H^*(\Si_h; \Z)$ for a flat bundle.  Since none of the examples obtained by stabilization are maximal, we ask: 

\begin{qu}
Does there exist a representation $\rho: \pi_1(\Si_h)\ra\Homeo(S^1)$ with maximal Euler class that lifts to $\Homeo(\Si_{g}^{1})$?  
\end{qu}

Since $\Diff^r_0(\D^2) \ra \Diff^r_0(S^1)$ is a homotopy equivalence for all $r$, the argument in the proof of Proposition \ref{prop:boundary-extend} gives no obstruction to extending any action on $S^1$ to an action on the disk. (Indeed for $r=0$, any group action can be extended by ``coning off'' the circle to the disk.  However, we will see examples of group actions by $C^1$ diffeomorphisms that cannot be lifted to $\Diff^1(\D^2)$ in Section \ref{sec:dynamical}.)   Geometric actions of surface groups provide interesting examples to study: 

\begin{prob}
The action of $\PSL(2,\R)$ by M\"obius transformations on $S^1$ extends to a smooth action on $\D^2$, so the representations $\pi_1(\Si_h)\ra\PSL(2,\R)\ra \Diff^\infty(S^1)$ discussed above extend to $\Diff^\infty(\D^2)$. 
\begin{enumerate}
\item[(a)] Classify all extensions of these actions. (Are there any exotic ones?)
\item[(b)] If $k>1$ divides $2h-2$, then $\pi_1(\Si_h)$ embeds as a lattice in the central extension of $\PSL(2,\R)$ by a cyclic group of order $k$, which also acts naturally on the circle.  (This action is simply the lift of the natural M\"obius action of $\PSL(2,\R)$ to the $k$-fold cover of $S^1$, which is also a circle.)   Does such an action of $\pi_1(\Si_h)$ on $S^1$ extend to $\Diff^\infty(\D^2)$?
\end{enumerate} 
\end{prob}

The actions described in (b) are called \emph{geometric} and share many properties with Fuchsian actions, including strong topological rigidity \cite{Mann-invent, Mann-Wolff}.   It would be interesting to know if they have obstructions to extensions over the disc.

\paragraph{Surface braid groups and point-pushing.} 
There is a surjection $\Mod_{g,m}^b\ra\Mod_g^b$ that forgets the marked points. The kernel is described by an exact sequence
\begin{equation}\label{eqn:exact-sequence}\pi_1(\Diff(\Si))\ra\Br_m(\Si)\xrightarrow{P}\Mod_{g,m}^b\ra\Mod_g^b\ra1,\end{equation}
where $\Br_m(\Si)$ is the \emph{surface braid group}, defined as the fundamental group $\pi_1(\Conf_m(\Si))$ of the space of unordered configurations of $m$ points on $\Si$. The homomorphism $P$ is called the \emph{point-pushing homomorphism}; see \cite{farb-margalit} for further details.

\begin{thm}[Bestvina--Church--Souto]\label{thm:bcs}The homomorphism $P$ in (\ref{eqn:exact-sequence}) is not realized by $C^1$ diffeomorphisms when $\Si$ is closed, $g\ge2$ and $m\ge1$. 
\end{thm}

We sketch Bestvina--Church--Souto's argument, which relies on the fact that the Euler class introduced above is a \emph{bounded} cohomology class \cite{milnor,wood}.

\begin{thm}[Milnor-Wood inequality]\label{thm:milnor-wood} 
Let $\Si_g$ be a closed oriented surface of genus $g\ge1$, and let $[\Si_g]\in H_2(\Si_g;\Z)$ be the fundamental class. 
For any $\rho:\pi_1(\Si_g)\ra\Homeo(S^1)$, the Euler class $e(\rho)\in H^2(\Si;\Z)\simeq\Z$ satisfies 
\[2-2g\le e(\rho)\le 2g-2.\]
Furthermore, if $\rho$ factors through $r:\GL_2^+(\R)\ra\Homeo(S^1)$ acting on rays through the origin in $\R^2$, then $|e(\rho)|\le g-1$. 
\end{thm}

\begin{proof}[Proof of Theorem \ref{thm:bcs}]
We sketch a proof in the case $m=1$. In this case $\Br_1(\Si_g)\simeq\pi_1(\Si_g)$. If $P$ is realized by diffeomorphisms we have the following diagram. 
\begin{equation}\label{eqn:bcs-diag}\begin{xy}
(0,15)*+{\Diff^1(\Si_{g,1})}="A";
(25,15)*+{\GL_2^+(\mathbb R)}="B";
(-20,0)*+{\pi_1(\Si_g)}="C";
(0,0)*+{\Mod_{g,1}}="D";
(25,0)*+{\Homeo(S^1)}="E";
{\ar "A";"B"}?*!/_3mm/{D};
{\ar "C";"A"}?*!/_3mm/{\til P};
{\ar "A";"D"}?*!/^3mm/{};
{\ar "C";"D"}?*!/^3mm/{P};
{\ar "B";"E"}?*!/_3mm/{r};
{\ar "D";"E"}?*!/^3mm/{\al};
\end{xy}\end{equation}
Here $D$ is induced by the action on the tangent space at the marked point and $r$ is the homomorphism from Theorem \ref{thm:milnor-wood}. The homomorphism $\al$ was originally defined by Nielsen; it comes from the isomorphism $\Mod_{g,1} \simeq\Aut(\pi_1\Si_{g})$ together with the action on the Gromov boundary $\pa\pi_1(\Si_g)\simeq S^1$ (or equivalently the boundary of the compactification of the universal cover $\tilde{\Sigma}\simeq\H^2$).  See \S \ref{sec:dynamical} for a description. The left-hand triangle commutes, but the right-hand square does not. However, it does commute up to homotopy on the level of classifying spaces \cite[\S3]{BCS}. The proof concludes by comparing the image of the Euler class $e\in H^2(B\Homeo(S^1))$ when pulled back to $H^*(\Si_g)\simeq\Z$ along the outer paths in the diagram. 

On the one hand, the composition $\al\circ P$  factors as $\pi_1(\Si_g)\ra\PSL(2,\R)\ra\Homeo(S^1)$, and it follows that $e(\al\circ P)\in H^2(\Si_g)$ is the Euler number of the unit tangent bundle $T^1\Si_g\ra\Si_g$ (as in the proof of Proposition \ref{prop:boundary-extend}). Thus $e(\al\circ P)=2-2g$. 
On the other hand, by the Milnor-Wood inequality (see Theorem \ref{thm:milnor-wood}), $| e(r\circ D\circ \til P)|\le g-1$. This contradiction implies that $\til P$ does not exist. 
\end{proof}

Topologically, $P:\pi_1(\Si_g)\ra\Mod_{g,1}$ is the monodromy of the trivial bundle with respect to the diagonal section $\Delta\subset\Si\times\Si$. Theorem \ref{thm:bcs} implies that this bundle has no flat connection such that the diagonal is \emph{parallel}, i.e.\ a leaf of the foliation. See also \cite{mv}.

\begin{qu}[Realizing point-pushing by homeomorphisms]
Can $P:\pi_1(\Si_g)\ra\Mod_{g,1}$ be realized by homeomorphisms? What about by homeomorphisms of the blow up of $\Si_{g,1}$ at the marked point (i.e.\ by homeomorphisms of $\Si_g^1$ that may be nontrivial on the boundary). If one tries to run the above proof of Theorem \ref{thm:bcs}, one must replace $\GL_2^+(\R)$ with $\Homeo(S^1)$ in (\ref{eqn:bcs-diag}), and then the Milnor--Wood inequality does not provide any contradiction. 
\end{qu}

\begin{qu}[Disc pushing]
If we replace the marked point by a boundary component, there is a ``disk-pushing" subgroup $D:\pi_1(T^1\Si_g)\hookrightarrow\Mod_g^1$ analogous to the point-pushing subgroup $P:\pi_1(\Si_g)\hookrightarrow\Mod_{g,1}$. Is $D$ realized by diffeomorphisms? The proof of Theorem \ref{thm:bcs} does not adapt to this case because the Euler class on $\Mod_g^1$ is trivial. 
\end{qu}

\begin{qu}[Handle pushing]
Consider an inclusion $\Si_g^1\hra\Si_{g+1}$ obtained by gluing on a genus-$1$ handle. This induces a homomorphism $\Mod_g^1\ra\Mod_{g+1}$. Is the composition $H:\pi_1(T^1\Si_g)\xra D\Mod_g^1\ra\Mod_{g+1}$ realized by diffeomorphisms? This homomorphism can be interpreted as ``handle-pushing". This problem is harder than disk-pushing since a realization does not have an obvious invariant submanifold (e.g.\ a point or a disk as before). 
\end{qu}

The second author proved the following generalization of Theorem \ref{thm:bcs} to many other locally symmetric manifolds \cite{tshishiku-nielsen}. 

\begin{thm}[Tshishiku] \label{thm:tshishiku-nielsen}
Let $G$ be a real simple noncompact Lie group. Assume $G$ is not isogenous to $\SO(n,1)$ for any $n\ge3$. Then there exists a lattice $\Ga<G$ so that for the associated locally symmetric space $M=\Ga\backslash G/K$, the point-pushing homomorphism $P:\pi_1(M)\ra\Mod(M,*)$ is not realized by diffeomorphisms. 
\end{thm}

The outline is the same as the proof of Theorem \ref{thm:bcs}, but it uses different obstructions, including Pontryagin classes, Chern--Weil theory, and Margulis superrigidity.  

\begin{qu}
Let $M$ be a finite volume hyperbolic 3-manifold. Is the point-pushing homomorphism $P:\pi_1(M)\ra\Mod(M,*)$ realized by diffeomorphisms? 
\end{qu}

\subsection{The MMM classes and Morita's nonlifting theorem} 

In this section, we discuss the question of whether or not the surjection $\pi:\Diff_\pa(\Si_{g,m}^b)\ra\Mod_{g,m}^b$ is split, i.e.\ whether or not the entire mapping class group is realized by diffeomorphisms (this is a special case of Problem \ref{prob:lifting-mod}). There are a few low-complexity examples when $\pi$ does split; for example when $\Si=\Si_1$ is a closed torus or when $\Si=\Si_{0,3}^1$ is a disk with 3 marked points; these examples will be discussed more in \S\ref{sec:positive}. 
The first negative result was proved by Morita \cite{Morita}, originally with the assumption $g\geq 18$.

\begin{thm}[Morita nonlifting]\label{thm:morita}
Let $\Si_g$ be a closed surface of genus at least 2. Then the surjection $\pi:\Diff^r(\Si_g)\ra\Mod_g$ does not split for any $r\ge0$. Furthermore, if $g\ge10$ then $\pi$ has no $C^2$ splitting over any finite-index subgroup of $\Mod_g$.   
\end{thm}

We sketch Morita's proof using the following theorem on the cohomology of $\Mod_g$. See \cite{wahl} for an exposition of the proof.

\begin{thm}[Morita, Harer stability, Madsen--Weiss] \label{thm:mw} In degrees $*\le 2(g-1)/3$, the cohomology $H^*(\Mod_g;\Q)$ is a polynomial algebra generated by certain \emph{tautological classes} $e_i\in H^{2i}(\Mod_g;\Q)$. As a characteristic class of a $\Si_g$-bundle $E\ra M$ over a manifold $M$, the class $e_i$ is defined as 
\[e_i(E)=\int_\Si e(T^vE)^{i+1}\in H^{2i}(M),\]
where $\mathbb R^2\ra T^vE\ra E$ is the \emph{vertical tangent bundle}, $e(T^vE)$ is the \emph{Euler class}, and $\int_\Si:H^*(E)\ra H^{*-2}(M)$ is the \emph{fiber-integration} (or Gysin) map. 
\end{thm}

Morita \cite{Morita} defined the characteristic classes $e_i\in H^*(\Mod_g)$, and showed they are often nonzero (by constructing interesting $\Si_g$-bundles). Harer \cite{Harer} proved that they are \emph{stably} nonzero, and Madsen--Weiss \cite{Madsen-Weiss} proved that these classes account for all of the cohomology in low degree (their techniques also give an alternate proof that $e_i\neq0$ for $g$ large). 
\begin{proof}[Proof of Theorem \ref{thm:morita}]
Let $\Diff^2(\Si_g)^\de$ denote $\Diff^2(\Si_g)$ with the discrete topology.  We may view $\pi$ as a composition of continuous maps 
$\Diff^2(\Si_g)^\de \overset{\alpha} \ra\Diff^2(\Si_g) \overset{\beta} \ra\Mod_g$. To prove the theorem, it suffices to show that 
\[\alpha^*\circ\beta^*:H^*(B\Mod_g;\Q)\ra H^*(B\Diff^2(\Si_g);\Q)\ra H^*(B\Diff^2(\Si_g)^\de;\Q)\]
has nontrivial kernel.    The result of Earle--Eells \cite{EE} cited in the previous section is that $\Diff^2(\Si_g) \to \Mod_g$ is a homotopy equivalence, so by classifying space theory the induced map on classifying spaces is as well, hence  $\beta^*$ is an isomorphism.   Thus, we need to show that $\alpha^*$ is not injective. Since $B\Diff^2(\Si)^\de$ classifies flat surface bundles, $\ker\alpha^*$ is the ideal of characteristic classes of $\Si_g$-bundles that vanish for flat bundles. Using Bott's vanishing theorem (a topological obstruction to existence of certain foliations), which applies because of the $C^2$ hypothesis, one checks that if $E\ra M$ is flat, then $e_i(E)=0$ for $i\ge3$. By  the Theorem \ref{thm:mw}, $e_3\in H^*(\Mod_g;\Q)$ is nonzero if $g\ge10$. Thus $e_3$ is a \emph{nonzero} element of $\ker\alpha^*$. 

If $j:\Ga\hra\Mod_g$ is finite index, then the induced map $j^*:\H^*(\Mod_g;\Q)\ra H^*(\Ga;\Q)$ is injective by transfer, so $j^*(e_3)\neq0$ and the same argument shows $\Ga$ is not realized by diffeomorphisms. 
\end{proof}

The proof above shows $\Mod_g$ is not realized by $C^2$ diffeomorphisms when $g\ge10$ 
As mentioned in the introduction, there are now many different proofs of Theorem \ref{thm:morita} \cite{BCS,FH,ST,markovic-saric} that show $\Mod_g$ is not realized by $C^1$ or $C^0$ diffeomorphisms (with varying improvements on the genus bound); in particular, both Theorem \ref{thm:bcs} and Theorem \ref{thm:st} have Theorem \ref{thm:morita} as a consequence. However, none of these alternate proofs give information about finite-index subgroups of $\Mod_g$.  In particular, Markovic--Saric's proof of non-realizaiton by homeomorphisms uses torsion elements in a crucial way, so does not answer the following.   

\begin{prob}[Robust non-lifting]
Show that no finite-index subgroup of $\Mod_g$ is realized by homeomorphisms for $g\ge2$. 
\end{prob}

A seemingly more difficult problem is to show non-lifting for the \emph{Torelli group}. 

\begin{prob}
Show that $\ca I_g$, the kernel of $\Mod_g\ra\text{Sp}_{2g}(\Z)$ is not realized by diffeomorphisms. 
\end{prob}

The question of realizing mapping class groups $\Mod_{g,m}^b$ when $m,b\ge1$ is also interesting. For example, the mapping class group $\Mod_{0,m}^1$ of a disk $\mathbb D^2$ with $m$ marked points $\mathbf{z}_m\sbs\mathbb D^2$ is the classical braid group $B_m$. The following theorem of Nariman \cite{Nariman2017} states that there is no cohomological obstruction to splitting $\Diff_\pa(\mathbb D^2-\mathbf{z}_m)\ra B_m$. Nevertheless, in contrast, the mapping class group of a disk is not realized by $C^1$ diffeomorphisms -- see Theorem \ref{thm:st} below. 

\begin{thm}[Nariman] \label{thm:Nariman}
Consider the surjection $\Diff^\infty(\D^2-\mathbf{z}_m)\ra B_m$. The induced map on group cohomology $H^*(B_m; A) \to H^*(\Diff^\infty(\D^2- \mathbf{z}_m); A)$ is split injective in all cohomological degrees and for all abelian groups A.
\end{thm} 
Nariman's work also considers another question related to subgroups of mapping class groups.  
 There is a natural ``geometric" map $\psi: \Br_{2g+2} \hookrightarrow \Mod_{g}^{2}$ induced by lifting mapping classes to a double cover $\Sigma_{g}^{2}$ of the disk $\D$ ramified over the points of $\mathbf{z}_{2g+2}$.   One description of this map is as follows:  Each $f \in \Diff(\D, \mathbf{z}_{2g+2})$ has a canonical lift to a homeomorphism of the cover $\Sigma_{g}^{2}$; this is the lift that fixes both boundary components pointwise.  This gives an injective map $\Psi: \Diff(\D, \mathbf{z}_{2g+2}) \to \Homeo(\Sigma_{g}^{2}, \partial \Sigma_{g}^{2})$, and the induced map on isotopy classes is $\psi$.  Nariman asked if $\Psi$ can be ``smoothed":  is there a homomorphism $\Diff(\D, \mathbf{z}_{2g+2}) \to \Diff(\Sigma_{g}^{2},  \partial \Sigma_{g}^{2})$ that induces $\psi$ on mapping class groups?   
Building on the techniques used in his proof of Theorem \ref{thm:Nariman}, Nariman shows that there is no cohomological obstruction.

\begin{thm}[Nariman \cite{Nariman2017}]
Let $g \geq 1$.  
There is a space $Y$, homologically equivalent to the classifying space $B \Diff^\delta (\Sigma_{g}^{2}, \partial \Sigma_{g}^{2})$ and a map $\Phi$ such that the following diagram is homotopy commutative
   \[
    \xymatrix{ B \Diff^\delta (\D^2,\mathbf{z}_{2g+2}) \ar[r]^{\Phi}  \ar[d] & Y \ar[d] \\
              B \Mod(\D^2, \mathbf{z}_{2g+2}) \ar[r]_{B\psi} & B \Mod(\Sigma_{g}^{2})  } 
\]
\end{thm} 

We give an alternative proof in \S\ref{sec:positive} by constructing an explicit lift.   However, the following question remains open.

\begin{qu} 
Can the ``geometric" braid group $\psi: \Br_{2g+2} \hookrightarrow \Mod_{g}^{2}$ described above be realized in $\Diff(\Sigma_{g}^{2})$?  
\end{qu}

%% file: Dynamical.tex
\section{Dynamical obstructions to realizations}\label{sec:dynamical}

This section showcases the use of dynamical tools in realization problems.   In many contexts, one can use the algebraic structure of a group $\Gamma$ to force some dynamical behavior -- such as the existence of fixed points or invariant sets -- when $\Gamma$ acts by diffeomorphisms on a manifold $M$.  If the action comes from a section of $\Diff^r_0(M) \to H$ for some group $H$, properties of $H$ often furnish additional dynamical information.   

The proofs of non-realizability of $\Mod_g$ by homeomorphisms due to Markovic and Markovic--Saric \cite{Markovic, MS} are the most sophisticated instance of this phenomenon.   The heart of Markovic's idea is to show that the realization of Dehn twists about simple closed curves have support on sets that, roughly, look like the annuli on which the Dehn twists are defined.  Once this is appropriately formalized, one uses relations in $\Mod_g$ to derive a contradiction.    As the argument is quite involved, we will not be able to do it justice here and instead refer the reader to Le Cavez's work \cite{LeCalvez}, which contains a nice exposition as well as some alternative techniques towards the proof. 

Here we focus on simpler examples that illustrate a tension between the algebraic structure of a group $\Gamma$ and the dynamics of a realization.  As a first basic illustration, we discuss non-realizability for groups of germs.

\subsection{Non-realizability of groups of germs}   \label{sec:germs}

Recall that $\mathcal{G}^r(n)$ denotes the group of germs at 0 of orientation-preserving $C^r$ diffeomorphisms of $\R^n$ fixing 0 (i.e. the quotient of $\Diff^r(\R^n, 0)$ by the subgroup of diffeomorphisms that restrict to the identity in a neighborhood of $0$), and $\Gamma \subset \mathcal{G}^r(n)$ is \emph{realizable by diffeomorphisms} if there is a section of  $\Diff^r(\R^n, 0) \to \mathcal{G}^r(n)$ over $\Gamma$.   Here, we prove the following non-realizability result, as a first step towards Problem \ref{prob:germ}.  

\begin{thm} \label{thm:germ_nonrealize}
There exists an (explicit) finitely-generated, torsion-free subgroup $\Gamma$ of $\mathcal{G}^0(n)$ that cannot be realized in $\Diff^0(\R^n, 0)$. 
\end{thm} 

As we remark in the proof, the subgroup $\Gamma$ can in fact be taken to lie in $\mathcal{G}^\infty(n)$.  

Theorem \ref{thm:germ_nonrealize} was proved in \cite{Mann15} for $n=1$.  Here, we give a simplified argument that proves the general case.   The argument uses essentially only one dynamical tool (which will make a later appearance in Section \ref{sec:Franks-Handel}). This is the {\em basin of attraction}.  

\begin{defi} 
Let $f$ be a homeomorphism of a manifold $M$.  A point $p \in M$ is \emph{attracting} if there exists a neighborhood $W$ of $p$ such that $\bigcap_{n>0} f^n(\overline{W}) = p$.  In this case the \emph{basin of attraction} of $f$ at $p$ is the set $\{x \in M \mid f^n(x) \to p \}$.   
\end{defi} 

It follows from the definition that the basin of attraction is $f$-invariant, and equal to $\bigcup_{n>0}f^n(W)$ for any sufficiently small neighborhood $W$ of the attracting fixed point -- so in particular is open.    Furthermore, if $f$ and $g$ commute, then $g$ permutes the attracting fixed points of $f$ and their basins of attraction.

\begin{proof}[Proof of Theorem \ref{thm:germ_nonrealize}]
We first define $\Gamma$ by specifying homeomorphisms of which we take the germs, then prove that there is no realization.  The key algebraic properties of $\Gamma$ are that it contains a Baumslag-Solitar subgroup $BS(1,2)=\langle a,b\mid aba^{-1}=b^2\rangle$, as well as several commuting elements.  

Parametrize $\R^n \setminus \{0\}= S^{n-1} \times \R$, with the sets $S^{n-1} \times \{t\}$ approaching $0$ as $t \to -\infty$.  All of the homeomorphisms in $\Gamma$ will be defined radially, so we start by working on $\R$.   A $BS(1,2)$ subgroup of $\Diff^\infty(\R)$ is given by $a:x\mapsto 2x$ and $b:x\mapsto x-1$.  To define the other elements,  let $Y \subset X$ be two $a$-invariant disjoint unions of intervals in $(-\infty, 0)$ and 
let $f_1$ be a diffeomorphism supported on $X$ that agrees with $b$ on $Y$ in some neighborhood of $-\infty$.  Let $f_2 = f_1^{-1} \circ b$, which, in a neighborhood of $-\infty$, is supported on the $a$-invariant set $\R \setminus Y$.  Let $g_i$ be arbitrary diffeomorphisms of $\R$ that commute with $a$ in a neighborhood of $-\infty$ and have support equal to $\fix(f_i)$.  

For each $h \in \{a, b, f_1, f_2, g_1, g_2\}$ the map $(s, t) \mapsto (s,  h(t))$ of $S^{n-1} \times \R$ extends to a homeomorphism of $\R^2$ fixing $0$. 
(By first conjugating $a$ and $b$ by a strong contraction at $-\infty$ as in Construction \ref{const:smoothing}, one may in fact take these to be germs of smooth diffeomorphisms, infinitely tangent to the identity at 0; in which case $f_i$ and $g_i$ may be chosen so that these define diffeomorphisms as well.)  Abusing notation slightly, we identify $h$ with the germ of the homeomorphism so defined.   
These germs satisfy the relations $aba^{-1} = b^2$, $f_1f_2=b$, and $[f_i, a] = [f_i, g_i]=1$.  

Let $\Gamma<\mathcal{G}^0(n)$ be the subgroup generated by $a, b, f_1, f_2, g_1$ and $g_2$, and suppose that $\phi: \Gamma \to \Homeo(\R^n, 0)$ is a realization.  Let $B$ denote the attracting basin for $b$.  The relation $aba^{-1} = b^2$ implies that $B$ is $a$-invariant.    Note that, if $U$ is a sufficiently small ball about the origin, then $B$ is the increasing union of the sets $\bigcup_{n>0}a^{-1}(U)$; each a homeomorphic image of a ball.  Using this (and the annulus theorem), we can parametrize $B \setminus 0$ by $S^{n-1} \times \R$, on which $\phi(b)$ acts by $(s, t) \mapsto (s, t-1)$.    We now claim that the attracting basin for $\phi(a)$ has closure contained in $B$, and will show this by finding an $a$-invariant set properly contained in $B$.  Take a fundamental domain $S^{n-1} \times [0,1]$ for $\phi(b)$; the relation $aba^{-1} =b^2$ implies that the image of any fundamental domain for $b$ under $a$ is a fundamental domain for $b^2$.   Now $\phi(a)(S^{n-1} \times [0,1])$ is compact so contained in $S^{n-1} \times [-M,M]$ for some $M$, hence
$\phi(a)^{-1}(S^{n-1} \times [-2M,2M])$ is a fundamental domain for $b^{2M}$ containing $[0,1]$, thus contained in $S^{n-1} \times [-2M,2M]$.  This gives an $a$-invariant compact set with closure contained in $B$, as desired.    

Finally, since $\phi$ is a realization, we known the germ of $\phi(g_i)$ at $0$: in a neighborhood of $0$ the homeomorphism $\phi(g_i)$ is supported on the disjoint union of an annulus $A \cong S^{n-1} \times I$ with its images under $\phi(a)^{n}$ for $n \in \N$.  Since $g_i$ and $a$ commute, the boundary of the basin of attraction of $a$ is equal to the set of accumulation points of the iterates $a^{-n}(A)$, for $n>0$.   Since $g_i$ and $f_i$ commute, this set of accumulation points is $f_i$-invariant.   But $f_1 f_2 = b$, contradicting the fact that there is no $b$-invariant set contained in $B \setminus \{0\}$.  
\end{proof} 

Despite the elementary nature of this example, the algebraic structure of groups of germs is far from understood.  As well as the problems raised in the introduction, we ask (expecting the answer to be positive):

\begin{qu}
Can $\mathcal{G}^r(n)$ and $\mathcal{G}^r(m)$ be distinguished by their finite subgroups?  What about finitely-generated, torsion-free subgroups?  
{\em Realizable} subgroups?
\end{qu}

\begin{qu}
Fix $n>1$.  Can $\mathcal{G}^r(n)$ and $\Diff^r(\R^n)$ be distinguished by their finitely-generated, torsion-free subgroups?  
\end{qu}

\subsection{$C^1$ techniques: linear representations and Thurston stability}  \label{sec:C^1}

In \cite{ghys91}, Ghys proved the following non-realization result for boundary maps.

\begin{thm}[Ghys]
There is no section of the natural restriction map $\Diff^1_0(\D^n) \to \Diff^1_0(S^{n-1})$.
\end{thm}
In the case when $n$ is even, the proof actually gives a finitely-generated subgroup $\Gamma \subset \Diff^\infty_0(S^{n-1})$ which has no realization in $\Diff^1_0(\D^n)$, we give a sketch below.  

\begin{proof}[Proof sketch -- even dimensional case] Supposing $n=2k$, identify $S^{n-1}$ with the unit sphere $\{(z_1, ... z_k) \in \C^k \mid \sum |z_i|^2 = 1\}$.   Let $\mu_1, \ldots, \mu_k$ be distinct $p^{th}$ roots of unity for some prime $p>n$, and define an order $p$ diffeomorphism $f: (z_1, \ldots, z_k) \mapsto (w_1 z_1, \ldots, w_k z_k)$.  One can explicitly write $f$ as a product of commutators $f = [a_1, b_1] \ldots [a_j, b_j]$ such that $a_i$ and $b_i$ all commute with $f$.    Let $\Gamma = \langle f \rangle \times \langle a_1,\ldots, b_j \rangle$.   Supposing that $\phi$ is an extension, $\phi(f)$ is a finite order element of $\Diff^1_0(\D^n)$ with no fixed points on $S^{n-1}$ hence has an interior fixed point (by Brouwer).  Smith theory says that the action of $\phi(f)$ on the doubled disc (i.e. $S^n$) has fixed set equal to a mod-$p$ homology sphere; since this fixed set is disconnected, it must be a $0$-sphere.   One concludes that $\phi(f)$ has a single fixed point, say $x \in \D^n$, preserved by $\phi(a_i)$ and $\phi(b_i)$ since these commute with $\phi(f)$.   Since $x$ is an isolated fixed point, the derivative $D\phi(f)_x$ is a linear map of order $p$ and full rank.   However, the centralizer of such a map in $\GL(n,\R)$ is abelian, so $D\phi(f)_x = \prod[D\phi(a_i)_x, D\phi(b_i)_x] = 1$, contradiction.   
\end{proof}

The tension here comes from forcing a global fixed point (through Smith theory and the algebraic structure of $\Gamma$), then using the constraints from the algebraic structure of $\Gamma$ on the linear representation given by taking derivatives at that point.    It would be interesting to see if this line of argument can be improved to even dimensional spheres -- because there is no fixed-point free finite order map of $S^{2k}$, Ghys' argument in this case is more involved and uses simplicity of the full diffeomorphism group.   Torsion elements also play an essential role in the argument, so the following question remains open in all cases except $n=2$. 

\begin{qu} \label{Q:SphereSection}
Does there exist a torsion-free, finitely-generated $G \subset \Diff_0(S^{n-1})$ with no section to $\Diff_0(\D^n)$?    
\end{qu}

The $n=2$ is covered in \cite{Mann13}, and parallels the strategy in Ghys' proof above.  \emph{Distorted elements} in $G$ are used as a substitute for torsion elements, and work of Franks and Handel on distorted surface diffeomorphisms is used as a substitute for Smith theory to force an isolated fixed point.   To complete the second half of the argument, one needs an element with nontrivial, finite-order, derivative at the fixed point; this is accomplished using rotation numbers.   

In the absence of a nontrivial linear representation, \emph{Thurston's stability theorem} gives an alternative -- and more powerful -- tool whenever one has a $C^1$ action with global fixed point.  

\begin{thm}[Thurston stability \cite{Thurston}] \label{thurston stab} 
Let $M$ be a manifold, $x \in M$, and let $G$ be finitely-generated group of germs of $C^1$ diffeomorphisms fixing $x$ and with trivial derivative at $x$.  Then there is a surjective homomorphism $G \to \Z$.  
\end{thm}

The proof uses a clever renormalization procedure to compare growth of higher order terms. A nice exposition can be found in \cite[\S 6.2]{Foliations}.  

\begin{app} \label{app:thurston_typical}
A typical application of Theorem \ref{thurston stab} is as follows: suppose $\Si$ is a surface and $x$ an accumulation point of fixed points for $G \subset \Diff^1(\Si)$.  Then the linear representation of $G$ by taking derivatives at $x$ either gives a homomorphism to $\Z$ via determinant, or via conjugation to the unipotent subgroup $\left( \begin{smallmatrix} 1 & \ast \\ 0 & 1 \end{smallmatrix} \right)$, or is trivial.  In the last case, Thurston stability produces a homomorphism to $\Z$.   
Thus, to apply this reasoning to a non-realization problem in $\Diff^1(\Si)$, one aims to force the existence of accumulation points of fixed points for a finitely-generated subgroup $G$ with $H^1(G; \Q) = 0$.      
\end{app}

As an illustrative example, here is a special case of a result of Salter-Tshishiku.  

\begin{thm}[Salter-Tshishiku  \cite{ST}]\label{thm:st}
Let $n \geq 5$.  There is no realization of $\Diff^1(\D^2, \mathbf{z}_n) \to B_n$. 
\end{thm} 

\begin{proof} 
Suppose $\phi: B_n \to \Diff^1(\D^2, \mathbf{z}_n)$ is a realization.  The commutator subgroup of $B_n$ is finitely generated, and by a theorem of Gorin--Lin, is perfect provided that $n \geq 5$.    Let $x$ be a point in the frontier of $\fix(\phi([B_n, B_n]))$, which is nonempty since $\phi(B_n)$ pointwise fixes $\partial \D^2$.  Now use the argument from Application \ref{app:thurston_typical} to derive a contradiction.  
\end{proof} 

Salter--Tshishiku actually prove the stronger result that, for any surface $\Si_{g,n}^b$ (possibly with $b=0$)
the surface braid subgroup of $\Mod(\Si_{g,n}^b)$ is not realized by $C^1$ diffeomorphisms whenever $n \geq 6$, then use the same technique to re-prove Morita's non-lifting Theorem \ref{thm:morita} in class $C^1$ for surfaces of genus $g \geq 2$.  
The idea is as follows:  Supposing that a realization $\Mod_g \to \Diff^1(\Sigma)$ exists, let $f \in \Mod_g$ be the hypereliptic involution.  The Lefschetz fixed point theorem implies that $\phi(f)$ has $2g+2$ fixed points; these are permuted by the centralizer of $f$.   Using Dehn twists about explicit curves, one produces a subgroup of the centralizer isomorphic to a quotient of $B_{2g+2}$, acting with a global fixed point, and applies a Thurston stability argument to derive a contradiction.  

We will show in \S\ref{sec:positive} that $B_3$ can be realized by (smooth) diffeomorphisms of a disc with three marked points; however, the following are open. 

\begin{qu}
Can $B_4$ be realized by diffeomorphisms?  
\end{qu} 

\begin{qu}
For $n>3$, can $B_n$ be realized by homeomorphisms of the marked disc?  
\end{qu} 

\begin{qu}
The argument of \cite{ST} does not generalize to finite index subgroups.  For $n \ge 4$, is there a finite index subgroup of $B_n$ that \emph{can} be realized by diffeomorphisms?  What about by homeomorphisms?    As a particular case, what about the {\em pure braid group} $P_n:=\ker[B_n\ra S_n]$ ?
\end{qu}


\subsection{Franks--Handel nonlifting} \label{sec:Franks-Handel}
Franks and Handel have a different proof of a general version of Morita's nonlifting Theorem \ref{thm:morita} for genus $\ge3$. 

\begin{thm}[Franks--Handel \cite{FH}]\label{thm:franks-handel}
Let $f \in \Mod_g$ preserve a genus $h \geq 1$ subsurface $S$, and be pseudo-Anosov on $S$ and isotopic to the identity on $\Sigma_g \setminus S$.  Let $H \cong \Mod(\Sigma_g \setminus S) \subset \Mod_g$ be the subgroup of mapping classes of diffeomorphisms pointwise fixing $S$; this centralizes $f$.  Provided that $\Sigma_g \setminus S$ has genus at least 2, $\langle f, H \rangle$ cannot be realized in $\Diff^1(\Sigma_g)$.  
\end{thm}

The proof uses two of the dynamical tools we have discussed: Thurston stability and basins of attraction.  The additional ingredients are a) the theory of \emph{prime ends} -- a compactification method compatible with group actions on open, connected, simply connected domains in the plane -- and b) the Thurston--Nielsen classification of mapping classes and the dynamics of their boundary maps, which we describe next. Notably, the proof does not use torsion in $\Mod(\Si)$. 

\paragraph{The Thurston--Nielsen classification.}
Building on work of Nielsen \cite{Nielsen}, Thurston \cite{Thurston-Nielsen} gave a classification of elements of the mapping class group of a surface.  Each $\phi\in\Mod_g$ has a representative homeomorphism which is either finite order, pseudo-Anosov, or reducible.   A homeomorphism $f$ being \emph{Pseudo-Anosov} means that the surface admits two $f$-invariant transverse measured foliations, whose transverse measures are expanded by $\lambda$ and $1/\lambda$ (respectively), for some $\lambda>0$.    The foliations may have finitely many singularities, each a $p$-pronged saddle, for some $p$.  
\emph{Reducible} means that there is a finite collection $\ca A$ of disjoint simple closed curves invariant under $f$ such that, after passing to a power that fixes each curve, $f$ preserves a small tubular neighborhood $N(\ca{A})$ of $\ca A$, and acts on each connected component of $\Si \setminus N(\ca A)$ separately either by the identity, a pseudo-Anosov homeomorphism, or possibly a Dehn twist if the component is an annulus.  The structure of reducible elements is discussed more in \S\ref{sec:raag}.

Any homeomorphism $f$ of $\Si$ can be lifted to a homeomorphism of the universal cover $\tilde \Si \cong \H^2$.  The lift is a \emph{quasi-isometry} of $\H^2$ so extends to a homeomorphism of the boundary $S^1$ of the Poincar\'e disc.  Thinking of the mapping class group as the group of outer automorphisms of $\pi_1(\Si)$; a choice of lift of $f$ fixes a choice of representative \emph{automorphism} corresponding to the mapping class of $f$, and the action of the lift on $S^1$ agrees with the action of the automorphism on the Gromov boundary of $\pi_1(\Si)$.    Thus, if $f$ and $g$ are isotopic, two isotopic lifts determine the same boundary action; and the Thurston--Nielsen classification above gives a helpful way to understand the boundary action of mapping classes, as will be used in the proof below.  
Thurston's survey \cite{Thurston-Nielsen} is a good introduction to this theory; more details can be found in \cite[\S13]{farb-margalit} and \cite{FLP}.

Using this machinery, we now sketch the proof of Franks and Handel's theorem.  Their exposition is focused on the dynamics of homeomorphisms of the disc with fixed points; we focus on the dynamics of pseudo-Anosov mapping classes applicable to this situation.   

\begin{proof}[Proof sketch of Theorem \ref{thm:franks-handel}]
Suppose for contradiction that $\phi: G = \langle f, H \rangle \to \Diff^1(\Sigma_g)$ is a realization.   
One may check directly from a presentation of $H$ that $H^1(H; \Z) = 0$, so the goal is to find an accumulation point of global fixed points for $\phi(H)$ and apply Thurston stability.  

We first lift $\phi$ to an action on the universal cover, and study its action by homeomorphisms on the compactification of $\widetilde{\Sigma_g} \cong \H^2$ as $\D^2$.    Let $a$ be a pseudo-Anosov representative of $f$ on $S$.  After passing to a power of $a$ (and $f$) if needed, we may assume $a$ fixes a point $x \in S$, and (after a further power) also fixes all leaves through $x$.  Choose a connected component $\widetilde{S}$ of the preimage of $S$ in $\H^2$ and take a lift $\tilde{a}$ fixing $\tilde{x} \in \widetilde{S}$. Then the boundary action of $\tilde{a}$ has a finite, even number of fixed points, alternately attracting and repelling, corresponding to the endpoints of the leaves of the (lifted) foliations that pass through $\tilde{x}$.  If $x$ is a non-singular point, there will be two attractors and two repellers, otherwise, for a $p$-pronged singularity, there are $p$ of each.  Let $\tilde{f}$ be the lift of $\phi(f)$ with boundary action agreeing with $\tilde{a}$; these lifts have the property that $\tilde{f}$ and $\tilde{a}$ are isotopic.  There is also a canonical choice of lift $\tilde{H}$ of $\phi(H)$ pointwise fixing $\partial \widetilde{S} \subset \partial \H^2$.  This will commute with $\tilde{f}$, giving a well-defined lift of $\phi$. 

The next step is to understand the dynamics of the action of $\tilde{f}$ on the interior of the disc.  Although $f$ is only isotopic to $a$, so \emph{a priori} it can have very different dynamics, we claim that the attracting points of $\tilde{f}$ on the boundary really are attractors, in the sense that there exists an open neighborhood $W \subset \D^2$ of each such point $p$ with the property that $\tilde{f}^n(W)$ converges to $p$.   To see this, we use the fact that $f$ is a quasi-isometry of $\D^2$, so there exists $K$ such that the image of each geodesic under $f$ is a $K$-quasi-geodesic.    Also, since $\tilde{f}$ and $\tilde{a}$ are isotopic, there exists $K'$ such that $d(\tilde{f}(x), \tilde{a}(x)) < K'$ (in the hyperbolic distance) for all $x \in \H^2$.    Recall also that $a$ expands length along the singular leaf $L$ ending at $p$ by some $\lambda > 1$.    Let $\gamma$ be a bi-infinite geodesic perpendicular to $a$ passing through $L$ at distance $t$ from $\tilde{x}$.    Then $\tilde{f}(\gamma)$ is a $K$ quasi-geodesic passing through a point at distance $K'$ from a point distance $\lambda t$ from $\tilde{x}$ along $L$, and it is easy to verify that whenever $t$ is sufficiently large, $\tilde{f}(\gamma)$ is contained inside the hemisphere cut out by $\gamma$.   This suffices to show that $p$ is an attracting point.  

Let $U$ denote the basin of attraction of $\tilde{f}$ at $p$ (as defined in Section \ref{sec:germs}).   Since $\tilde{f}$ and $\tilde{H}$ commute, and $H$ fixes $p$, the basin $U$ is $\tilde{H}$-invariant. An easy improvement of this argument shows that, in fact, for every $\tilde{h} \in \tilde{H}$, there exists $n \in N$ such that 
for each attractor or repeller $z$ of $\tilde{f}$ in $\partial \H^2$, the map $\tilde{f}^n\tilde{h}$ has $z$ as an attractor (or repeller) with the same basin at $z$ as $f$.  

  \begin{figure}[h!]
     \labellist 
  \small\hair 2pt
  \pinlabel $U$ at 45 45
   \pinlabel $p$ at 80 -7 
   \pinlabel $y$ at 100 75
   \pinlabel $q$ at 155 35
   \endlabellist
  \centerline{
    \mbox{\includegraphics[width=1.1in]{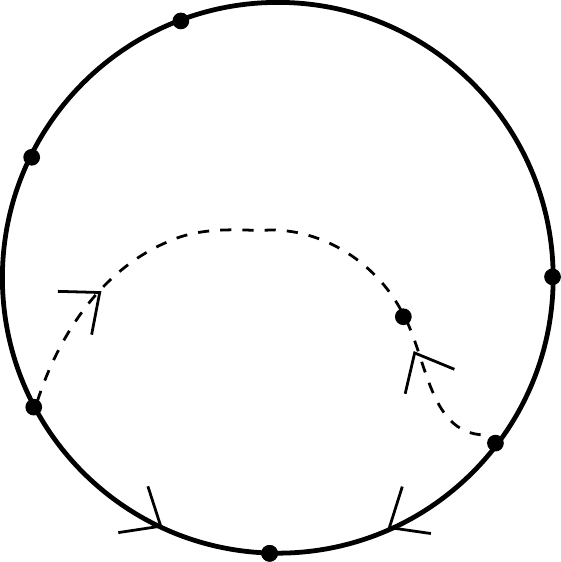}}
    \vspace{10pt} }
\caption{\small Dynamics of $\tilde{f}$ on $\overline{\H^2}$}
\label{pAdisc}
     \end{figure}
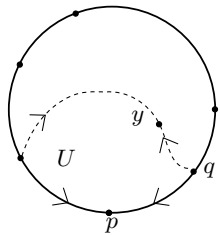

Now we find a point $y \in \H^2$ fixed by any such homeomorphism $\tilde{f}^n \circ \tilde{h}$.   As an easy case, suppose $U$ were an open ball with boundary topologically an embedded circle $C$.  Let $y$ be the closest point of $\fix(\tilde{f}^n \circ \tilde{h}) \cap C$ to the right of the repelling point $q$ adjacent to $p$, as indicated in the figure.  Using the fact that $f$ and $H$ commute, one shows this ``first point" for $\tilde{f}^n \circ \tilde{h}$ is actually independent of $h$, and so fixed by each such $\tilde{f}^n\tilde{h}$.  When $h = \id$, we may take $n=0$, so this shows it is also a fixed point of $\tilde{f}$, and hence of $\tilde{H}$ as well.  
In general, $U$ is not necessarily a homeomorphic disc with circle boundary, but one can show that $U$ is open, connected, and simply connected, and thus replace $\partial U$ in the argument above with the \emph{prime end compactification} of $U$ to finish the argument in an analogous manner.  

This gives us a single fixed point for $\phi(H)$ in $\Sigma_g$.   To produce infinitely many more, one repeats the argument above using other powers of $a$.   It is a standard fact (a consequence of Poincar\'e recurrence) that pseudo-Anosov diffeomorphisms have infinitely many periodic points -- in fact, a dense set in the surface.  See \cite[\S14]{farb-margalit}.   These are partitioned into (infinitely many distinct) \emph{Nielsen classes}, where $x$ and $y$ are said to be equivalent if there exists a lift of $a^n$ to $\H^2$ that fixes both $x$ and $y$.   The argument we gave above produced a global fixed point $y$ for $\phi(H)$ in the same Nielsen class as the original periodic point $x$ of $a$.   Thus, starting with points from different Nielsen classes for $a$ will produce infinitely many distinct fixed points for $\phi(H)$, which must accumulate somewhere on $\Sigma$.  This is what we needed to show.  
\end{proof}


\subsection{Obstructions to realizing non-discrete groups: forcing continuity}  \label{sec:continuity}

We return to discussing realization problems for large (non finitely-generated) subgroups under the ``restrict-to-boundary" map. 
Groups of \emph{quasiconformal} homeomorphisms are an interesting case.  Quasiconformal mappings are those with ``bounded deviation'' from conformal maps: as a special case, if $f$ is a differentiable map defined on an open domain $U \subset \C$, $f$ is quasiconformal if $\sup_{p \in U} \frac{|f_z(p)| + |f_{\bar z}(p)|}{|f_z(p)| - |f_{\bar z}(p)|}$ is finite.  There are other formulations of the definition which apply to homeomorphisms of arbitrary metric spaces.   The pseudo-Anosov maps from the Thurston--Nielsen classification above give examples of quasiconformal homeomorphisms between Riemann surfaces; indeed one approach to proving the classification (due to Bers \cite{Bers}) is through the theory of quasiconformal maps.  

A quasiconformal homeomorphism of a surface with boundary restricts to what is called a {\em quasi-symmetric} homeomorphism of $S^1$.  Thus, there is a natural surjection $\QC(\D^2) \to \QS(S^1)$.  The section problem for this map was answered negatively by Epstein and Markovic.  

\begin{thm}[Epstein--Markovic \cite{EM}] \label{thm:EM}
There is no section of the restriction map $\QC(\D^2) \to \QS(S^1)$.  
\end{thm} 

As explained in \cite{EM} (and perhaps part of their original motivation) there is a nice connection between this problem and the Nielsen realization problem for $\Mod_g$.  If $\QC(\D^2) \to \QS(S^1)$ had a section {\em with the additional property} that the extension over $\PSL(2,\R)$ were the identity, then the mapping class group $\Mod(\Sigma_g)$ could be realized by homeomorphisms, as follows:   Let $G$ be the group of all boundary mappings of lifts of homeomorphisms of $\Sigma_g$ to $\D^2 = \overline{\H^2}$.  Since the representative diffeomorphisms from the Thurston--Nielsen classification of mapping-classes are quasiconformal, these boundary maps are all quasi-symmetric. 
There is an exact sequence $\pi_1(\Sigma_g) \to G \to \Mod(\Sigma_g)$.  Assuming a section as above, let $\phi$ be a lift of $G$ to $\QC(\D^2)$; then $\pi_1(\Sigma_g)$ acts by M\"obius transformations.   The quotient of $\phi(G)$ by $\phi(\pi_1(\Sigma_g))$ then gives a realization of $\Mod(\Sigma_g)$ in the group of quasiconformal homeomorphisms of $\Sigma_g$.

A major step in Epstein and Markovic's proof of Theorem \ref{thm:EM} is to show that any such realization would be forced to be {\em continuous}, and so they ask whether every section of the boundary map $\Homeo(\D^2) \to \Homeo(S^1)$ is necessarily continuous.  
Note that at least one section exists, namely from ``coning off" the circle to the disc.   
More generally, one can ask the following.

\begin{qu} Let $M$ be a manifold with boundary, fix $r>0$, and suppose that $\Diff^r_0(M) \to \Diff^r_0(\partial M)$ admits a section.  Is that section necessarily continuous?  
\end{qu}

Only two results are known in this direction, one in $C^\infty$ and one in $C^0$ regularity.  Both prove something much stronger than the answer to the question above, but both use their respective assumptions on regularity in an essential way. 

\begin{thm}[Hurtado \cite{Hurtado}] \label{thm:Hurtado}
Let $M$ and $N$ be closed manifolds, and $\phi: \Diff^\infty_0(M) \to \Diff^\infty_0(N)$ any homomorphism.  Then $\phi$ is continuous.  
\end{thm}

\begin{thm}[\cite{Mann16}]
Let $M$ be a closed manifold and $G$ a separable topological group.  Any homomorphism $\phi: \Homeo_0(M) \to G$ is continuous.  
\end{thm}

We conclude this section with three additional open questions.  

\begin{qu}[Epstein--Markovic]
Fix $n>2$.  Does the map $\QC(\D^n) \to \Homeo(S^{n-1})$ have a section over its image?  
\end{qu}

\begin{qu}
Let $M$ be a manifold with boundary.  If $\Diff^r_0(M) \to \Diff^r_0(\partial M)$ admits a section, what can one say about the topology of $M$?  Are there examples where the action of $\Diff^r_0(\partial M)$ does not preserve a foliation?   
\end{qu}

For $\Diff_0^\infty$ this question was asked by Ghys \cite{ghys91}, who also asks specifically whether the map admits a section when $M$ is a handlebody.

%% file: Positive1.tex
\section{Positive results}\label{sec:positive}

\subsection{Automorphisms, blow-ups and smoothings}
The mapping class group of the torus $\Mod_{1,1} \cong \Mod_1 \cong \SL(2,\Z)$ is realized by linear diffeomorphisms of $\T^2 := \R^2 / \Z^2$ (with 0 as a marked point). Automorphism groups of translation surfaces are a natural generalization of this example.  

A {\em translation surface} is a surface with the additional structure of a holomorphic 1-form; equivalently, it is a surface obtained by taking a finite collection of polygons in $\R^2$ and identifying parallel sides in pairs by translations. Such a surface inherits a singular Euclidean structure with finitely many cone singularities, each with angle a positive multiple of $2\pi$.  An (orientation preserving) {\em affine automorphism} of a translation surface $X$ is a homeomorphism of $X$ that permutes its singularities and is locally in $\Aff^+(\R^2)$ at all other points.  These form a group, denoted $\Aut^+(X)$.  

While the automorphism group of the square torus includes both $\SL(2,\Z)$ and all translations in $\R^2/\Z^2$, when a translation surface $X$ has at least one singularity, the map $\Aut^+(X) \to \Mod(X)$ is always injective.  By definition $\Aut^+(X)<\Mod(X)$ is realized by homeomorphisms, but not obviously by {\em diffeomorphisms}.  In this section, we discuss techniques that allow one to promote an action by homeomorphisms to a smooth action, and use this to realize affine automorphisms and other groups.

\paragraph{Realizing $\Aut^+(X)$ by diffeomorphisms.} 
Let $X$ be a translation surface with holomorphic 1-form $\omega$.  Assume $X$ is not a flat torus, so has at least one singularity.   Taking the linear parts of the defining affine maps gives a homomorphism $\Aut^+(X) \to \SL(2,\R)$ whose image, denoted $\SL(X, \omega)$, is a discrete subgroup.  This subgroup may be finite covolume or not, but is never cocompact.  (For a gentle introduction, see \cite{Schwartz}.)
We show how to realize $\Aut^+(X)$ by diffeomorphisms, under the simplifying assumption that $\SL(X, \omega)$ is finitely generated.  

In this case, $\SL(X, \omega)$ has a presentation of the form 
\[ \langle a_1, b_1, \ldots, a_g, b_g,  c_1, \ldots, c_m,  d_1, \ldots, d_n  \mid  [a_1, b_1] \cdots [a_g, b_g]c_1 \cdots c_m d_1 \cdots d_n,  (c_i)^{q_i} \rangle \] 
where $q_i \in \N$,  $n \geq 1$ and $g, m \geq 0$. 
Let $p_1, \ldots, p_k$ denote the singularities of $X$.  The stabilizer of $p_i$ is a finite index subgroup of $\SL(X, \omega)$, so finitely generated and of the form above.    The open surface $X \setminus \{p_1, \ldots, p_k\}$ can be compactified by adding a circle of ``straight line directions" at each end, giving a surface $\overline{X}$ with $k$ boundary components, to which the action of $\Aut^+(X)$ extends.  

If $p_1$ has cone angle $2\pi M$, its stabilizer in $\SL(X, \omega)$ acts on the circle of directions by lifts of the linear action of $\SL(2,\R)$ on the circle of rays from the origin in $\R$ (topologically a circle) to the $M$-fold cover of this circle.  Other elements of the stabilizer in $\Aut^+(X)$, if any, act by rotations of order $k/M$, commuting with the lifted linear action.  

This action can be smoothly isotoped to one where, using the presentation above, all the $a_i$, $b_i$, and $d_1, \ldots, d_{n-1}$ act trivially,  and the $c_i$ and $d_n$ lie in $\SO(2)$ (we do not change the action of elements that map to the identity in $\SL(X, \omega)$).   To do this, work within the cyclic extension of $\SL(2,\R)$ by $\Z/M\Z$ in which these elements lie, and isotope $a_i, b_i$, and $d_i$, smoothly towards the identity, while isotoping each $c_i$ (through the space of order $q_i$ elements) to rotations.  The product $[a_1, b_1] ... [a_g, b_g]c_1 ... c_m d_1 ... d_{n-1}$  varies smoothly through the isotopy, so can be taken to define an isotopy of $d_n^{-1}$.   This isotopy of the action allows one to ``cone off" the circle to a disc, extending the action of the stabilizer smoothly over the disc.    A variation of this argument is given in detail in the next example.  

If $\Aut^+(X)$ acts transitively on the singularities, this is enough to determine a smooth action of $\Aut^+(X)$ on a closed surface obtained by gluing discs to each boundary component of $\overline{X}$ (which is topologically the surface $X$).  If not, repeat the procedure for each orbit of the action.  \qed

\paragraph{A realization of $B_3$ by diffeomorphisms.} Thurston observed that the braid group $B_3$ could be realized by homeomorphisms.   Building on our work above, we describe his construction and prove the following.  
\begin{thm} \label{thm:B_3-realize}   There is a section of $\Diff^\infty_\pa (\D, \mathbf{z}_3) \ra \Mod(\Si_{0,3}^1) = B_3$. 
\end{thm} 

\begin{proof}[Proof of Theorem \ref{thm:B_3-realize}]
Consider the linear action of $\SL(2,\Z)$ on $\T^2=\R^2/\Z^2$. The center $\{\pm1\}$ of $\SL(2,\Z)$ acts on $\T^2$ as the hyperelliptic involution; it has four fixed points $\{0\}\cup\{p_1,p_2,p_3\}$, which $\SL(2,\Z)$ permutes (later we will use that these points are $p_1=(0,\frac{1}{2})$, $p_2=(\frac{1}{2},\frac{1}{2})$, $p_3=(\frac{1}{2},0)$, when viewed in $\R^2$). The point 0 is a global fixed point and $\SL(2,\Z)$ acts on ${\bf p}=\{p_1,p_2,p_3\}$ transitively. The quotient $\PSL(2,\Z)=\SL(2,\Z)/\{\pm1\}$ acts on the quotient $\T^2/\{\pm1\}$, which is a sphere with four cone points, each of angle $\pi$.  (In the language of the previous section, this sphere is a \emph{half translation surface}, or Riemann surface equipped with a \emph{quadratic differential}.) 

Since $0$ is a global fixed point, we can compactify $\T^2/\{\pm1\} \setminus \{0\}$ by adding a circle of directions at $0$, obtaining an action of $\PSL(2,\Z)$ on a disk with 3 marked points $\Si_{0,3}^1=(\D^2,{\bf p})$.   

To relate $B_3$ to $\PSL(2,\Z)$, note that there is an exact sequence $1\ra\Z\ra B_3\ra\PSL_2(\Z)\ra1$. 
On the level of mapping class groups, the projection $B_3\ra\PSL(2,\Z)$ is equal to the map 
\[B_3\simeq\Mod_{0,3}^1\ra\Mod_{0,4}^0\simeq(\Z/2\Z\ti\Z/2\Z)\rtimes\PSL(2,\Z)\ra\PSL(2,\Z)\] induced by the map $\Si_{0,3}^1\ra\Si_{0,4}^0$ that collapses the boundary component of to a point (see \cite[Prop.\ 2.7]{farb-margalit}). On the group level $\PSL(2,\Z)\simeq\Z/2\Z*\Z/3\Z=\langle x,y: x^2=1=y^3\rangle$, and $B_3$ has a presentation $\langle a,b: a^2=b^3\rangle$, with $B_3\ra\PSL(2,\Z)$ the obvious homomorphism. 
Thus, we can view the $\PSL(2,\Z)$ action on $(\D^2,{\bf p})$ as a non-faithful action of $B_3$. 
To obtain a smooth realization of $B_3$, we will (1) isotope the action of $x,y\in\PSL_2(\Z)$ on the $\pa\D^2$ to the identity, preserving the relation $x^2=y^3$ to get an action by homeomorphisms, and (2) glue discs in at singularities, as we did for translation surfaces above, to smooth the action.

For (1), attach an annulus $A\simeq S^1\ti[0,1]$ to $\pa \D^2$ along $S^1 \times \{0\}$, enlarging the disc.  Under suitable parameterization, $a\in B_3$ acts on $\pa \D^2$ by a standard order two rotation, and $b$ by an order 3 projectively linear map.  Let $b_t$, $0 \leq t \leq 1/2$  be a smooth path of conjugates of $b$ through $\PSL(2,\R)$ such that $b_0 = b$, and $b_{1/2} \in \SO(2)$.   Now extend this to a smooth path through $\SO(2)$ for $1/2 \leq t \leq 1$, with $b_{1-\epsilon} = \id$ for all small $\epsilon$.   Let $a_t$ be a smooth path in $\SO(2)$ from $a_0 = a$ to $a_1 = \id$ such that $a_t^2 = b_t^3$ for all $t$.  Defining $a,b$ on $S^1\ti\{t\}\sbs A$ to agree with $a_t,b_t$ gives an extension of the action to a smooth action on the annulus that is identity in a neighborhood of $S^1 \times \{1\}$.  

For (2), as in the argument for translation surfaces, compactify $\D^2 \setminus \mathbf{p}$ by a circle $C_i$ of straight line directions at each $p_i$.  Observe that the action of $\PSL(2,\Z)$ on $C_1\cup C_2\cup C_3$ is the action of $\PSL(2,\Z)$ on the projectivized unit tangent circles at $p_1,p_2,p_3$ (which we may identify using the natural trivialization of $T\D^2$). Under our identification we have
\begin{align}  \label{eq ab}
\begin{split}
& a:C_2\ra C_2 \text{ is the identity, } a: C_1 \to C_3 \text{ agrees with } \left( \begin{smallmatrix} 0 & -1 \\ 1 & 0 \end{smallmatrix} \right) \in \PSL(2, \R),\\
& b: C_i \to C_{i+1} \text{ agrees with } \left( \begin{smallmatrix} 0 & -1 \\ 1 & 1 \end{smallmatrix} \right) \in \PSL(2, \R). 
\end{split}
\end{align}
Glue a disc $D_i$ to each $C_i$ and fix smooth collar neighborhoods of $C_i\subset D_i$ parametrized by $C_i \times [0,1]$. To extend to this collar, we can choose smooth isotopies of $a$ and $b$ from the maps in \eqref{eq ab} to the identity, and extend the action to 
$\cup_i C_i\times\{t\}$ by $a_t, b_t$, so that $a_0, b_0$ agree with rigid translations on $\cup_i C_i\times\{0\}$.  Then we can extend the action to the rest of $D_i$ by rigid translations.  The result is an action by diffeomorphisms of $a$ and $b$ on a disc, fixing the boundary pointwise, and permuting the centers of $D_1,D_2$ and $D_3$, which we take as three marked points. The reader may easily check that the mapping classes of $a$ and $b$ agree with the standard generators of $B_3$, as required.  
\end{proof}

Building on this line of argument, we also give a positive answer to Nariman's question mentioned at the end of Section \ref{sec:Cohomological}.

\begin{thm}  \label{thm:Nariman_q} 
Let $\mathbf{z} = \mathbf{z}_{2g+2}$ for some $g \geq 1$.  
There is a continuous homomorphism $\psi: \Diff(\D, \mathbf{z}) \to \Diff(\Sigma_{g}^{2}, \partial \Sigma_{g}^{2})$ that induces the geometric homomorphism $B_{2g+2} \to \Mod_{g}^{2}$ on mapping class groups.  
\end{thm} 
While Nariman showed there is no cohmological obstruction to this lift, the existence of the realization is somewhat surprising, given a few of the results we have already mentioned.  Specifically, Theorem \ref{thm:st} implies that $\psi$ cannot be obtained by a map that factors through $B_{2g+2}$, and Theorem \ref{thm:Hurtado} (which has a variation for manifolds that are not closed) also implies that such a map $\psi$ should essentially be continuous, with further work of Hurtado implying that its restriction to the subgroup $\Diff_c(\D, \mathbf{z})$ of diffeomorphisms fixing a neighborhood of $\mathbf{z}$ (which we know to be nontrivial by \cite{ST}) must be obtained by embedding copies of covers of the open, punctured disc into $\Sigma_{g}^{2}$.  This suggests, at least vaguely, that $\psi$ would have to be obtained by branching the punctured disc over $\mathbf{z}$, an inherently non-smooth construction.    

The proof uses two constructions, one a trick for smooth gluing, and the other a classical ``blow up" procedure similar to the technique above to deal with singularities of translation surfaces.  

\begin{const}[Smoothing actions glued on a codimension 1 submanifold] \label{const:smoothing}
Let $G$ be a group acting by smooth diffeomorphisms on manifolds $S_1$ and $S_2$.   Let $X_1$ and $X_2$ be diffeomorphic connected components of $\partial S_1$ and $\partial S_2$ respectively, and let $S$ be the manifold obtained by gluing $S_1$ and $S_2$ by a diffeomorphism $X_1 \to X_2$.   If, for each $g \in G$,  the action of $g$ on $X_1$ agrees with that on $X_2$ under the identification used in the gluing, then there is an obvious induced action of $G$ on $S$ by homeomorphisms.  However, this is \emph{conjugate} to an action by \emph{smooth diffeomorphisms} on $S$.  
The conjugacy can be obtained by a map $f: S \to S$ which is the identity outside a tubular neighborhood of the glued boundary components, and in the tubular neighborhood (identified with $X \times [-1,1]$, with the glued boundary components at $X \times \{0\}$) is locally a very strong contraction at $0$; taken strong enough so that $f g f^{-1}$ becomes infinitely tangent to the identity in the direction transverse to the boundary at $X \times \{0\}$ .  Details are worked out in \cite{Parkhe} using the local contraction $(x, y) \mapsto (x, e^{\frac{-1}{e^{-1/|y|}}})$.  
\end{const}

\begin{const}[Blow-up] \label{const:blow_up}
Let ${\bf p} = \{p_1, p_2, ..., p_k \}$ be a finite set of points in a manifold $M^n$.  The (oriented) \emph{blowup of $M$ at ${\bf p}$ } is a smooth manifold $\hat M$ obtained from $M$ by replacing each $x\in{\bf p}$ with the space of directions $T^1_xM\simeq S^{n-1}$ in its tangent space. There is a natural projection $\Phi: \hat{M} \to M$ that is a diffeomorphism away from ${\bf p}$, and a natural injection $\Phi^*:\Diff^r(M,{\bf p})\ra\Diff^{r-1}(\hat M)$ for $r\ge1$. In particular, any $C^1$ action on $(M,{\bf p})$ induces an action on $\hat M$ by homeomorphisms. 
\end{const}

\begin{proof}[Proof of Theorem \ref{thm:Nariman_q}]
Our strategy is to first build a map $\phi: \Diff(\D, \mathbf{z}) \to \Diff(\D, \mathbf{z})$, which has image in a subgroup that acts on a given neighborhood of $\mathbf{z}$ by translations, and also induces the identity map $B_n \to B_n$. 
Building this map is the bulk of the construction.  Given such a map $\phi$, the diffeomorphisms in its image can then be lifted to \emph{diffeomorphisms} of a cover branched over $\mathbf{z}$ as described above.  

To start, apply Construction \ref{const:blow_up} to blow up $\D$ at the set $\mathbf{z}$.  The new surface obtained (call it $D_0$) has $n+1$ boundary components, one corresponding to the original boundary $\partial$ of the disc, and the others corresponding to the blown up points.  

Enumerate $\mathbf{z} = \{z_1, z_2, ..., z_n\}$ and for $i = 1, 2, ..., n$, let $D_i$ be a blow-up of $\D$ at $\{z_i\}$.   Glue each $D_i$ to $D_0$ along the blow-up of $z_i$, using the identity map on the space of tangent directions at $z_i$.   The result is an $(n+1)$-holed sphere.  Now embed this $(n+1)$-holed sphere into $\D$ with the boundary component $\partial$  mapping to $\partial \D$.   The result is pictured in Figure \ref{gluing fig}; boundary components of the $D_i$ are labeled by their images under the map from the blow-up construction.  

  \begin{figure*}
   \labellist 
  \small\hair 2pt
   \pinlabel $z_1$ at 222 118
   \pinlabel $z_2$ at 315 120
   \pinlabel $z_3$ at 405 118
   \pinlabel $D_0$ at 270 250
   \pinlabel $\partial$ at 380 55
     \pinlabel $D_1$  at 0 220  
     \pinlabel $D_2$ at 272 140
     \pinlabel $D_3$ at 375 138
   \pinlabel $z_1$  at 35 113
   \pinlabel $\partial$ at 80 150
    \pinlabel $z_2$ at 25 189
   \pinlabel $z_3$ at 75 200
   \endlabellist
     \centerline{ \mbox{
 \includegraphics[width = 2.4in]{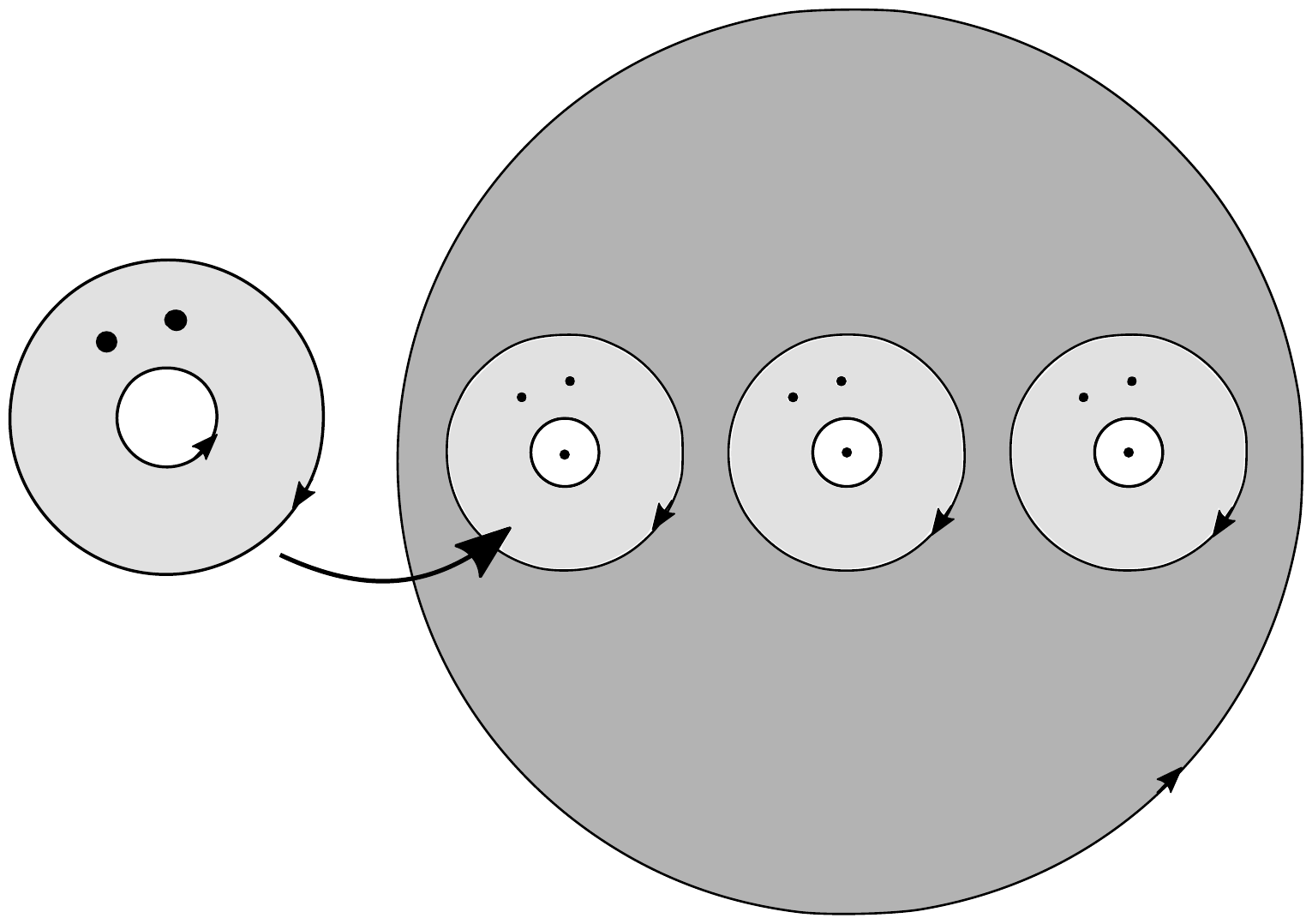}}}
 \caption{Gluing copies of $\D$ blown up at one point of $\mathbf{z}$ into a copy of $\D$ blown up at $\mathbf{z}$}
  \label{gluing fig}
  \end{figure*}
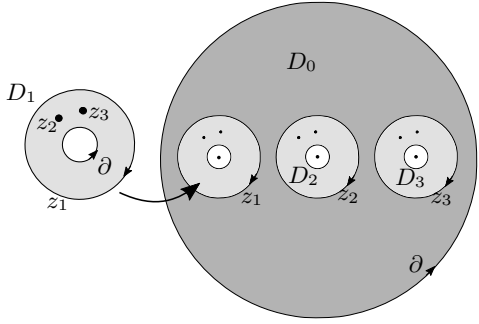

Let $C_1, C_2, ... C_n$ be the connected components of the complement of the image of the embedding (shown in white on the figure).  We may arrange the embedding so that each $C_i$ is a round disc of radius $\epsilon$, centered around the marked point $z_i$ on the original disc $\D$ (shown as midpoints of the white regions of the figure).  Construction \ref{const:blow_up} gives a natural homomorphism from $\Diff(\D, \mathbf{z})$ to $\Diff(D_i)$.  These actions of $\Diff(\D, \mathbf{z})$ on the various $D_i$ for $i =0, 1, ..., n$ agree on their glued boundary components, so Construction \ref{const:smoothing} produces a homomorphism from $\Diff(\D, \mathbf{z})$ to the diffeomorphisms of the $n+1$ holed sphere that was obtained by gluing the $D_i$ together.  We identify this surface with the image of its embedding in $\D$.    Since elements of $\Diff(\D, \mathbf{z})$ fix a neighborhood of $\partial \D$ pointwise, we may also arrange the embedding so that this action permutes the boundaries of the complementary discs $C_i$ by \emph{rigid translations}.  Thus, the action naturally extends to an action on $\D$ by diffeomorphisms, permuting the discs $C_i$ by translations.  In particular, the set of midpoints of the $C_i$ is preserved, so this action is by elements of $\Diff(\D, \mathbf{z})$.   Let $\phi: \Diff(\D, \mathbf{z}) \to \Diff(\D, \mathbf{z})$ denote this action.  

Finally we check that $\phi$ induces the identity map on the quotient $B_n = \Diff(\D, \mathbf{z})/\Diff_0(\D, \mathbf{z})$. By design of our blow-up, gluing, and embedding, if $f(z_i) = z_j$, then $\phi(f)$ maps $D_i$ to $D_j$, hence maps $C_i$ to $C_j$, and its center point $z_i$ to $z_j$. 

We may now quickly finish the proof.  As explained in \S 2, there is an injective map $\Psi: \Diff(\D, \mathbf{z}) \to \Homeo_\pa(\Sigma_{g}^{2})$.  Consider the map  $\Psi \circ \phi: \Diff(\D, \mathbf{z}) \to \Homeo_\pa(\Sigma_{g}^{2})$, which agrees with $\Psi$ on mapping class groups.  Each diffeomorphism in the image of $\phi$ has trivial (i.e. constant $\equiv id$) derivative in a neighborhood of each $z \in \mathbf{z}$, so its image under $\Psi$ is smooth everywhere.   Thus, $\Psi \circ \phi$ gives the desired map $\Diff(\D, \mathbf{z}) \to \Diff_\pa(\Sigma_{g}^{2})$.   
\end{proof}

%% file: RAAGs.tex

\subsection{Realizing free abelian and right-angled Artin groups}\label{sec:raag}

Recall from \S\ref{sec:Franks-Handel} the Nielsen--Thurston classification of elements of $\Mod(\Si)$ (finite order, reducible, pseudo-Anosov).  We now discuss reducible elements in more detail.  By definition, a reducible $\phi \in \Mod(\Si)$ admits a \emph{reduction system}, a finite collection $\ca A$ of disjoint isotopy classes of simple closed curves that is invariant under $\phi$.  The intersection of all {\em maximal} reduction systems is called the \emph{canonical reduction system} $\ca A_\phi$.   

Birman--Lubotzky--McCarthy \cite{blm} use reduction systems to show that all solvable subgroups of $\Mod(\Si_{g,n}^b)$ are virtually abelian, and give a (sharp) bound on the torsion-free rank of abelian subgroups.   
Farb \cite[Ch.\ 2, \S6.3]{farb-problist} remarks (without proof) that their perspective can be used to show that all free abelian subgroups $\Z^n<\Mod_{g,m}^b$ can be realized by homeomorphisms.    In essence, since the image of a reduction system $\ca A$ for $\phi$ under some element $\psi$ is a reduction system for $\psi \phi \psi^{-1}$,  whenever $\Gamma < \Mod(\Si_{g,n}^b)$ is abelian, then the union $\bigcup_{\phi \in \Gamma} \ca A_\phi$ is a reduction system for each element of $\Gamma$, and restrictions of commuting elements to subsurfaces are necessarily either powers of the same pseudo-Anosov, or of the same Dehn twist, or possibly the identity.  From there (proceeding with some care on annular neighborhoods of the reducing system) one can choose commuting representative homeomorphisms for elements of $\Gamma$.

A natural next step is to generalize this strategy to \emph{right-angled Artin subgroups} of $\Mod(\Si)$. 
A right-angled Artin group $A(G)$ is defined by a finite undirected graph $G = (V, E)$ via 
\[A(G)=\langle x\in V \mid [x,y]=1\text{ if }\{x,y\}\in E\rangle.\] 
Hence, free groups (which are always realizable) correspond to totally disconnected graphs, and free abelian groups to complete graphs.  
There are many other interesting RAAG subgroups of mapping class groups, see \cite[\S8]{koberda} and references therein.  

\begin{qu}
Is every right-angled Artin subgroup $A(G)<\Mod(\Si)$ realized by homeomorphisms? 
\end{qu}

We remark that the reduction-system strategy above seems to go through in some special cases; for instance, if $A(G)<\Mod(\Si)$ is such that each $x\in V$ {\em fixes} each curve in $\ca A_x$.  However, there are many RAAG subgroups that do not have this property.